\documentclass[11pt]{amsart}
\usepackage{amsmath,amsfonts,amsthm,amssymb, mathtools}
\usepackage{mathrsfs, graphicx,color,latexsym, tikz, 
multicol
}
\usepackage{color}
\usepackage{xcolor}
\usetikzlibrary{plotmarks}
\definecolor{xdxdff}{rgb}{0,0,0}
\definecolor{qqqqff}{rgb}{0,0,0}
\definecolor{uuuuuu}{rgb}{0,0,0}
\definecolor{ududff}{rgb}{0,0,0}
\usepackage[foot]{amsaddr}
\usepackage[e]{esvect}
\usetikzlibrary{shadows}
\usetikzlibrary{patterns,arrows,decorations.pathreplacing, calc}
\usetikzlibrary{arrows.meta}
\newtheorem{thm}{Theorem}[section]
\newtheorem{lem}[thm]{Lemma}
\newtheorem{cor}[thm]{Corollary}

\newtheorem{Def}{Definition}[section]
\newtheorem{prop}[thm]{Proposition}

\newtheorem{prob}{Problem}

\newtheorem{conj}{Conjecture}

\newcommand{\field}[1]{\mathbb{#1}}

\newcommand{\N}{\field{N}}
\newcommand\scalemath[2]{\scalebox{#1}{\mbox{\ensuremath{\displaystyle #2}}}}
\newcommand{\diam}{\mathrm{diam}}

\begin{document}
\title{\bf  The Hoffman program of graphs: old and new}
\author{Jianfeng Wang$^{1}$}
\address[Corresponding Author]{$^{1}$School of Mathematics and Statistics, Shandong University of Technology, Zibo 255049, China}
\email{jfwang@sdut.edu.cn}
\author{Jing Wang$^2$}
\address{$^2$Department of Applied Mathematics, Northwestern Polytechnical University, Xi'an, 710072, China}
\email{jwang66@aliyun.com}
\author{Maurizio Brunetti$^3$}
\address{$^3$Dipartimento di Matematica e Applicazioni, Universit\`a `Federico II', Naples (Italy)}
\email{maurizio.brunetti@unina.it}


\maketitle
\begin{abstract}

The Hoffman program with respect to any real or complex square matrix $M$ associated to a graph $G$ stems from A. J. Hoffman's pioneering work on  the limit points for the spectral radius of  adjacency matrices of graphs less than $\sqrt{2+\sqrt{5}}$. The program consists of two aspects: finding all the possible limit points of $M$-spectral radii of graphs 
and detecting all the connected graphs whose $M$-spectral radius does not exceed a fixed limit point. In this paper, we summarize the results on this topic  concerning several graph matrices, including the adjacency, the Laplacian, the signless Laplacian, the Hermitian adjacency and skew-adjacency matrix of  graphs. As well, the tensors of hypergraphs are  discussed.  Moreover, we obtain new results about the Hoffman program with relation to the $A_\alpha$-matrix. Some further problems on this topic are also proposed. \\

\end{abstract}

\baselineskip=0.2in

\section{\large Introduction}

All graphs considered here are simple, undirected and finite.  For a graph $G= (V(G),E(G))$, let $M(G)$ be a corresponding real or complex square matrix defined in a prescribed way. The {\it $M$-polynomial} of $G$ is defined as {\rm det}$(\lambda{I} - M(G))$, where $I$ is the identity matrix. The {\it
$M$-spectrum} of $G$ is the multiset ${\rm Sp}_M(G)$ consisting of the eigenvalues of $M(G)$, and the largest absolute value of them is called the {\it
$M$-spectral radius}  of $G$. We denote it by $\rho\!_{_M}\!(G)$. For $v \in V(G)$, $d(v)$ denotes the degree of a vertex $v$, $D(G)={\rm diag}(d(v_1),d(v_2),\dots,d(v_n))$ is the degree matrix of $G$, and $\Delta(G)$ is the maximum vertex degree in $G$.  

The adjacency matrix $A(G)$, the Laplacian matrix $L(G)=D(G)-A(G)$ and the signless Laplacian matrix $Q(G)=D(G)+A(G)$ stand among the most widely studied graph matrices.  For $M \in \{A, L, Q\}$, the $M$-eigenvalues are all real and the $M$-spectral radius is equal to the $M$-index, i.e.the largest $M$-eigenvalue.

In order to explain what the Hoffman program for graphs is, we recall that 
a real number $\gamma(M)$ is said to be an  {\it $M$-limit point} of the $M$-spectral radius of graphs if there exists a sequence of graphs $\{G_k\, |\, k\in \mathbb{N}\}$ such that  $$\rho\!_{_M}\!(G_i) \neq \rho\!_{_M}\!(G_j) \quad \text{whenever  $i \neq j$}, \quad \text{and}  \quad \lim_{k \rightarrow \infty}\rho\!_{_M}\!(G_k)=\gamma(M).$$
The Hoffman program consists of two steps: i)  establishing whether any $M$-limit point exists and determining all the possible values; ii) finding all the connected graphs whose $M$-spectral radius does not exceed a fixed $M$-limit point.

Hoffman program historically originated from two sources. The first one is the investigation carried out by A. J. Hoffman himself
\cite{hoffman} on the limit points of the $A$-index. The second source also goes back to the early seventies, and is connected to a geometric problem involving equiangular lines and the corresponding root systems  independently examined by 
Smith \cite{smith} and Lemmens and Seidel \cite{lem-sei}.

In this paper, we first survey the results on the Hoffman program scattered in literature. Then, we establish some new results on limit points for the $A_\alpha$-matrix of graphs introduced in \cite{niki} by Nikiforov.  

In Sections 2-4, several types of graphs will be mentioned  to acknowledge past achievements concerning the Hoffman program for $A$,$L$, and $Q$. Among them, we find the path $P_n$, the cycle $C_n$, the star $K_{1,n-1}$ and the graphs depicted in Fig.~1, where we assume that 
$c \geq b \geq a \geq 1$ for the $T$-shape tree $T_{a,b,c}$, and $c \geq a \geq 1$ for the $H$-shape tree $Q_{a,b,c}$. 

Section 5-8 also have a survey flavour and  are respectively devoted to collect what is known on the limit points of the Hermitian adjacency matrix ${\mathcal H} (G)$ associated to a mixed graph $G$; of the  $\{0, \pm 1\}$-matrix associated to a signed graph; of the skew-adjacency matrix associated to oriented graphs; and of the symmetric tensor associated to uniform hypergraphs.

Section 9 contains new results on the $A_{\alpha}$-limit points. In particular, we compute the $A_{alpha}$-limit points of some families of compound graphs. In another paper, we are going to show how these results can be employed to find out all the smallest $A_{\alpha}$-limit points larger than $2$.

For any fixed $\rho \geqslant 0$, the symbols  $\mathcal{G}^{\rho}_M$, $\mathcal{G}^{\leqslant \rho}_M$,   $\mathcal{G}^{<\rho}_M$,   and $\mathcal{G}^{>\rho}_M$ will respectively denote the set of connected graphs whose $M$-spectral radius is  equal, not exceeding, less than, and larger than $\rho$.

\begin{figure}[h]
\begin{center}
\includegraphics[width=0.9\textwidth]{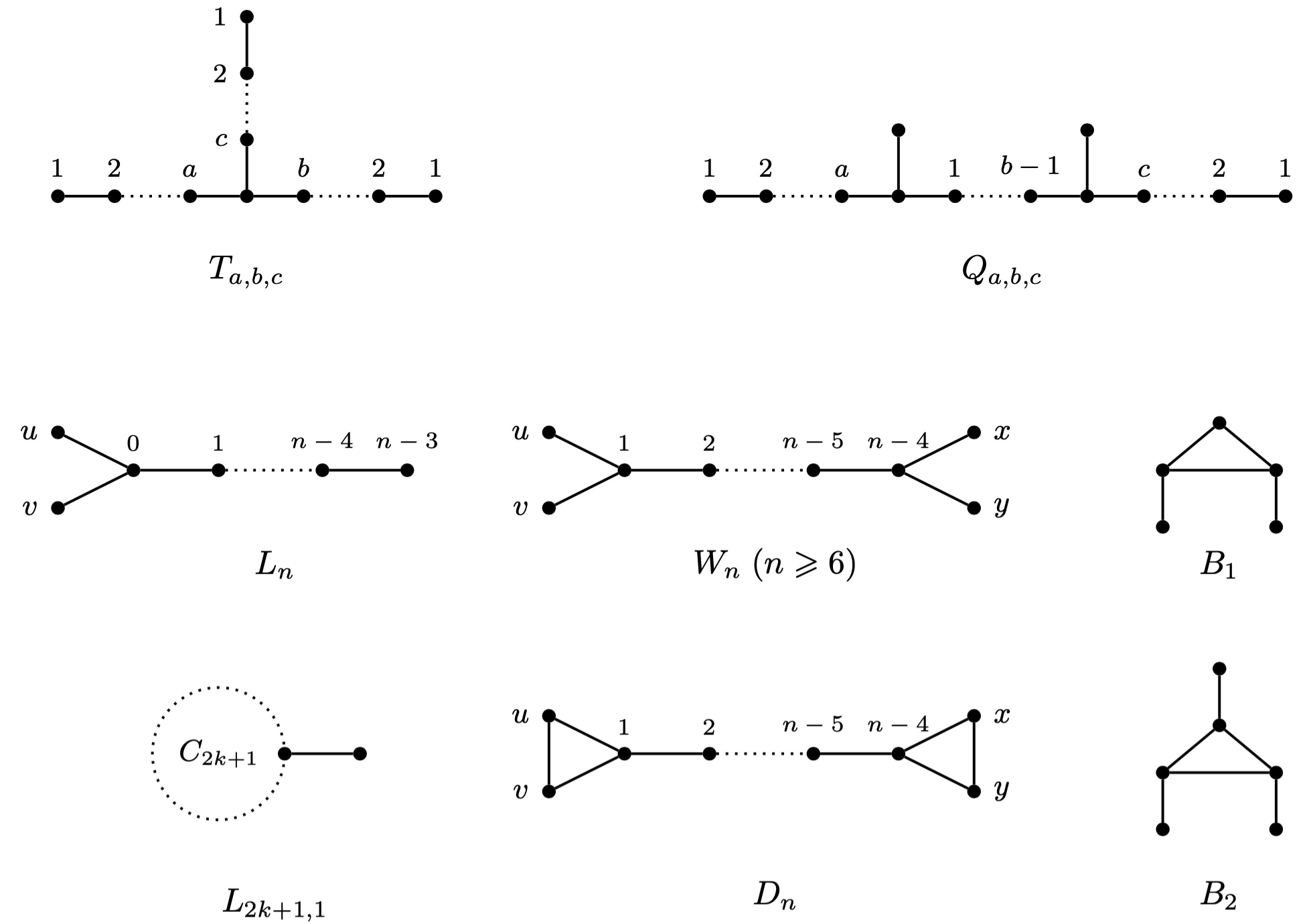}
\caption{ \label{fig1}  \small Some graphs used in the paper.}
\end{center}
\end{figure}

\section{Adjacency matrix}

The Hoffman program was initially carried out with respect to the adjacency matrix. Hoffman \cite{hoffman} contributed with the following pioneering result.

\begin{thm}{\rm \cite[Hoffman's theorem]{hoffman}}\label{hoffman} Let $\tau$ denote the number $(\sqrt{5}+1)/2$.
For $n \in \N$, let $\eta_n=\beta_n^{\frac{1}{2}}+\beta_n^{-\frac{1}{2}}$, where
$\beta_n$ is the positive root of
\[
\phi_n(x)=x^{n+1}-(1+x+x^2+\cdots+x^{n-1}).
\]
The numbers $2=\eta_1<\eta_2<\cdots$ are  the only $A$-limit points of the $A$-spectral radius of graphs smaller than
$\lim\limits_{n\rightarrow \infty}\eta_n = \tau^{\frac{1}{2}}+\tau^{-\frac{1}{2}}= \sqrt{2+\sqrt{5}}.$
\end{thm}

As recalled in Section~1, the $A$-spectral radius of a graph $G$ is also the largest eigenvalue of $A(G)$. This is due to the fact that $A$ is non-negative  (see \cite[Theorem 0.2]{cve-book}).

Encouraged by Hoffman, Shearer \cite{she} subsequently determined all the remaining $A$-limit points. In fact, he proved the following theorem.
\begin{thm}{\rm \cite{she}}\label{she-lim}
For any $\lambda \geq \sqrt{2+\sqrt{5}}=2.058+$, there exists a sequence of graphs $\{G_k\, |\, k\in \mathbb{N}\}$ such that $\lim\limits_{k \rightarrow \infty}\rho_A(G_k) = \lambda$.
\end{thm}

The graphs with $A$-index at most $\sqrt{2+\sqrt{5}}$ was characterized step-by-step in \cite{smith,hoffman,CDG,BN}. More precisely, Smith \cite{smith} determined all the connected graphs whose $A$-index is not greater than 2. They are now known in the literature as Smith graphs. After \cite{hoffman}, Cvetkovi\'c et al. determined the structure of graphs with $A$-index between 2 and $\sqrt{2+\sqrt{5}}$ in \cite{CDG}. Their description was completed a few years later by Brouwer and Neumaier \cite{BN}. We summarize such achievements
in the following theorem.
\begin{thm}\label{22+5} Let $\rho_1= \sqrt{2 +\sqrt{5}}$. The three sets  $\mathcal{G}^{<2}_A$, $\mathcal{G}^{2}_A$ and $\mathcal{G}^{ >2}_A \cap \mathcal{G}^{<\rho_1}_A$ can be described as follows.
\begin{itemize}
\item[$\mathrm{(i)}$]{\rm \cite{smith}} $\mathcal{G}^{<2}_A = \{ P_n,\; T_{1,1,n}, \; T_{1,2,c} \mid n \in \N, \; 2 \leqslant c \leqslant 4 \}$.
\\[-.6em]
\item[$\mathrm{(ii)}$]{\rm \cite{smith}}
$\mathcal{G}^{2}_A = \{ C_{n+2}, \; W_{n+5} \mid n \in \N \} \cup \{ K_{1,4}, \; T_{2,2,2}, \; T_{1,2,5}, \; T_{1,3,3} \}$.\\[-.6em]
\item[$\mathrm{(iii)}$]{\rm \cite{BN, CDG}}  $\mathcal{G}^{ >2}_A \cap \mathcal{G}^{<\rho_1}_A$ is the disjoint union of 
the subsets 
$$ \mathcal T_1= \{ T_{1,2,n+5}, \; T_{1,n+2,m+3}, \; T_{2,2,n+2} \mid n\in \N, \; m \geqslant n-1\} \cup \{T_{2,3,3} \}.$$
and
\[\label{mink}  {\mathcal T}_2= \{ Q_{1,1,2}, \; Q_{2,4,2}, \; Q_{2,5,3}, \; Q_{3,7,3}, \; Q_{3,8,4} \} \cup \{ Q_{a,b,c} \mid (a,c) \in \N^2 \setminus (1,1), \; b \geqslant b^*(a,c)  \}, 
\]
where  $ b^{*}(a, c) = \begin{cases}  a+c+2, \qquad \text{ for $a > 2$},\\
c+3,\qquad \quad \;\;\, \text{ for $a = 2$},\\
c,\qquad \qquad \quad \;\, \text{ for $a =1$}.\\
\end{cases}
$ 
\end{itemize}
\end{thm}

Woo and Neumaier \cite{woo} characterized the structure of the graphs whose $A$-index is between $\sqrt{2+\sqrt{5}}$ and $\frac{3}{2}\sqrt{2}$. We recall that  an {\it open quipu} is a tree of maximum vertex degree $3$ such that all vertices of degree 3 lie on a path; a {\it closed quipu} is a connected graph  of maximum vertex degree $3$ containing just one cycle $C$, and all vertices of degree 3 lie on $C$; finally,  a {\it dagger} is a path with a $3$-claw attached to an end vertex.

\begin{thm} {\rm \cite{woo}} \label{woo}
Let $\rho_1= \sqrt{2+\sqrt{5}}$ and $\rho_2= \frac{3}{2}\sqrt{2}$. A graph in $\mathcal{G}^{> \rho_1}_A  \cap  \mathcal{G}^{<\rho_2}_A $ is either an open quipu, a closed quipu, or a dagger.
\end{thm}

All daggers are in $\mathcal{G}^{< \rho_2}_A$. On the contrary, many open quipus and closed quipus have spectral radii greater than  $\frac{3}{2}\sqrt{2}$, and the structural conditions ensuring whether a quipu $T$  is in  $\mathcal{G}^{< \rho_2}_A$ or not are still to be completely determined. In any case, restrictions on the diameters of quipus belonging to $\mathcal{G}^{< \rho_2}_A$ are given by the following theorem.

\begin{thm} {\rm \cite{LL}} Let $\rho_2= \frac{3}{2}\sqrt{2}$. If an open quipu with $n\geqslant 6$ vertices belongs to $\mathcal{G}^{< \rho_2}_A$, then its diameter is at least $(2n-2)/3$. If a closed quipu with $n \geqslant 13$ vertices belongs to $\mathcal{G}^{< \rho_2}_A$, then
its diameter lies in the interval $(n/3, (2n-2)/3]$. 
\end{thm}

\begin{figure}[h]
\begin{center}
\includegraphics[width=0.9\textwidth]{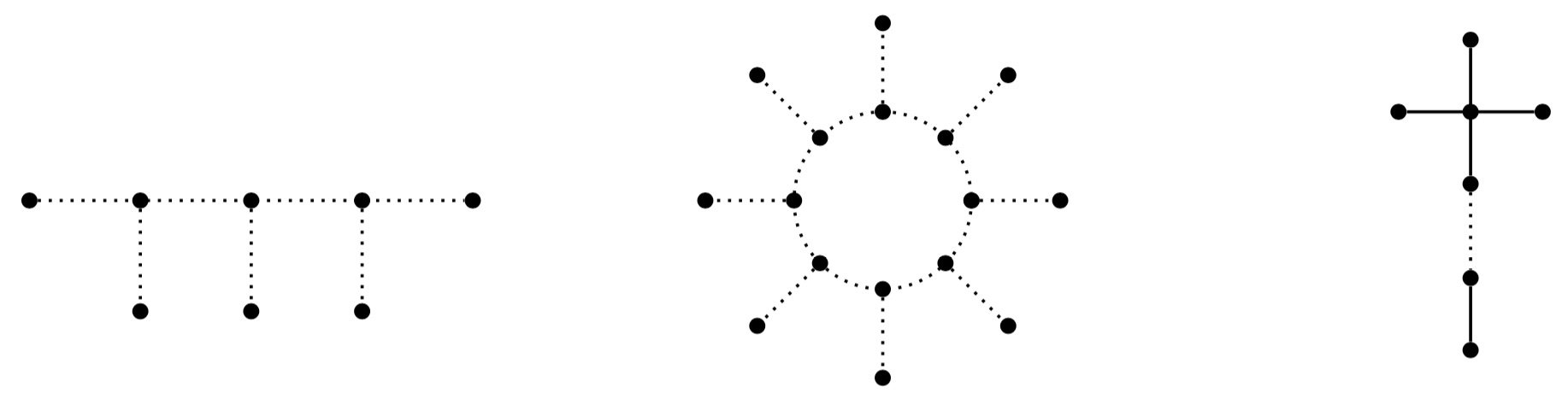}
\caption{ \label{fig2}  \small From left to right, an open quipu, a closed quipu and a dagger}
\end{center}
\end{figure}

\section{Laplacian matrix}

Guo obtained a Hoffman-like theorem for the $L$-spectrum in  \cite{guo-lim}. More precisely, he obtained the $L$-limit points of the $L$-spectral radius of graphs less than $2+\omega+\omega^{-1}=4.38+$,  where $\omega = \frac{{\;}1{\;}}{3}\left((19 + 3\sqrt{33})^{\frac{{\;}1{\;}}{3}} + (19 - 3\sqrt{33})^{\frac{{\;} 1{\;}}{3}}+1 \right)$.

\begin{thm}{\rm \cite{guo-lim}}\label{GLlimit}
Let $\beta_0 = 1$ and $\beta_n (n \geq 1)$ be the largest positive root of $$f_n(x) = x^{n+1}-(1+x+\cdots+x^{n-1})(\sqrt{x}+1)^2.$$ Let
$\alpha_n = 2+ \beta_n^{\frac{1}{2}} + \beta_n^{\frac{-1}{2}}$. Then,
$$4 = \alpha_0 < \alpha_1 < \alpha_2 < \cdots$$ are all the limit points
of $L$-spectral radius of graphs less than $\lim\limits_{n\rightarrow\infty}\alpha_n=2+\omega+\omega^{-1}$.
\end{thm}

Until now, the $L$-limit points of $L$-spectral radius of graphs no less than $2+\omega+\omega^{-1}$ have not been investigated. 
We here propose a conjecture. If true, it would be the Laplacian-counterpart of Shearer's result on $A$-limit points.

\begin{conj}
For any $\sigma \geq 2+\omega+\omega^{-1}$, there exists  a sequence of graphs $\{G_k\, |\, k\in \mathbb{N}\}$ such that $\lim\limits_{k \rightarrow \infty}\rho_L(G_k) = \sigma$.
\end{conj}

Graphs with relatively small $L$-spectral radius  were independently studied by several scholars. 
Omidi \cite{omi-4,omi-5} characterized the connected graphs whose $L$-index does not exceed $(5+\sqrt{13})/2$ (note that $2 +  \omega+\omega^{-1} = 4.38+ > (5+\sqrt{13})/2 =4.302+$). Later on, Simi\'c, Huang, Belardo and the first author of this paper \cite{WB-L4} studied the spectral determination of disjoint union of graphs with $L$-index at most $4$.
 Wang, Belardo and Huang \cite{WB-L4.5} also characterized the graphs whose $L$-index lies either in $[4,2+\sqrt{5}]$, $(2+\sqrt{5},2+\omega+\omega^{-1}]$ or $(2+\omega+\omega^{-1},4.5]$. Next theorem outlines in more details the results summarized in this paragraph. 
The last sentence in its statement  depends on results concerning the $Q$-index recalled in Section 4 and proved in \cite{WB-L4.5}.
\begin{thm}\label{Lindex4.5} Let $\tau_1=2+\sqrt{5}$ and $\tau_2=2+\omega+\omega^{-1}$.  The following equalities of sets hold. 
\vspace{.6em}
\begin{itemize}
\item[$\mathrm{(i)}$]{\rm \cite{omi-4,WB-L4}}  $\mathcal{G}^{< 4}_L = \{ P_n, C_{2n+1} \mid n \in \N \}$.\\[-.65em]
\item[$\mathrm{(ii)}$]{\rm \cite{omi-4,WB-L4}}
$\mathcal{G}^{4}_L = \{ K_{1,3},K_{1,3}+e,K_4-e,K_4 \} \cup \{C_{2k} \mid k \geq 2\}$.\\[-.65em]
\item[$\mathrm{(iii)}$]{\rm \cite{WB-L4.5}}
$\mathcal{G}^{> 4}_L \cap \mathcal{G}^{\leqslant  \tau_1}_L = \{T_{1,1,n-3}, \;
L_n \mid n \geq 5\}$;\\[-.65em]
\item[$\mathrm{(iv)}$]{\rm \cite{WB-L4.5}}
$\mathcal{G}^{>  \tau_1}_L \cap \mathcal{G}^{\leqslant \tau_2}_L = \bigcup_{i=1}^5 \mathcal U_i$, where
$$ \mathcal U_1 = \{ B_1, B_2\}; \qquad \mathcal U_2 = \{ L_{2k+2,1} \mid k \geq 2\}; \qquad  \mathcal U_3 = \{ T_{1,b,c} \mid c \geq b \geq 2\};
$$
$$ \mathcal U_4 = \{ Q_{a,b,c} \mid b \geq a+c+1\};  \qquad \mathcal U_5 = \{  W_n, D_n-xy, D_n \mid n \geq 8\}.
$$
\end{itemize}

\vspace{.8em} \noindent
Moreover, the set $\mathcal{G}^{> \tau_2}_L \cap \mathcal{G}^{\leqslant  4.5}_L$
just contains open and closed quipus.

\end{thm}

\section{Signless Laplacian matrix}

Inspired by Hoffman's theorem and Guo's Theorem \ref{GLlimit}, the first authors of this paper et al.\  \cite{wang-lim} 
determined the $Q$-limit points smaller than $2+\varepsilon$, where $\varepsilon=\frac{{\;}1{\;}}{3}\left((54 - 6\sqrt{33})^{\frac{{\;}1{\;}}{3}} + (54 + 6\sqrt{33})^{\frac{{\;} 1{\;}}{3}} \right)$. Note that $\varepsilon=\omega + \omega^{-1} =2.38+$, where $\omega$ is the number defined Section~3. The $L$-limit points and the $Q$-limits points less than $2+\varepsilon$ are the same. This is not surprising, since 
the proof of  \cite[Theorem~3.1]{wang-lim} consists in a reduction to trees, and it is well-known the $L$- and $Q$-spectra of bipartite graphs are equal.

\begin{thm}{\rm \cite{wang-lim}}\label{Qlimit}
Let $\beta_0 = 1$ and $\beta_n (n \geq 1)$ be the largest positive root of $$f_n(x) = x^{n+1}-(1+x+\cdots+x^{n-1})(\sqrt{x}+1)^2.$$ Let
$\alpha_n = 2+ \beta_n^{\frac{1}{2}} + \beta_n^{\frac{-1}{2}}$. Then
$$4 = \alpha_0 < \alpha_1 < \alpha_2 < \cdots$$ are all the limit points
of the $L$-index and  the $Q$-index of graphs less than {\small $\lim\limits_{n\rightarrow\infty}\alpha_n=2+\varepsilon$},
where $\varepsilon = \frac{{\;}1{\;}}{3}\left((54 - 6\sqrt{33})^{\frac{{\;}1{\;}}{3}} + (54 + 6\sqrt{33})^{\frac{{\;} 1{\;}}{3}} \right) = 2.38+$.
\end{thm}

So far, the $Q$-limit points of the $Q$-index which are not less than $2+\varepsilon$ have not yet been identified. As for the correspondent problem in the $L$-context, we state the following conjecture.

\begin{conj}
For any $\theta \geq 2+\varepsilon$,  there exists a sequence of graphs $\{G_k\, |\, k\in \mathbb{N}\}$ such that $\lim\limits_{k \rightarrow \infty}\rho_Q(G_k) = \theta$.
\end{conj}

The graphs with $Q$-index not exceeding $2+\varepsilon$ were gradually characterized in \cite{cve1,wangsl1,BW-Conj}. More in details, Cvetkovi\'c, Rowlinson and Simi\'c \cite{cve1} characterized the graphs in $\mathcal{G}^{\leqslant  4}_Q$. Afterwards, Wang et al. \cite{wangsl1}  started to describe the graphs in  $\mathcal{G}^{>  4}_Q \cap  \mathcal{G}^{\leqslant  2+\sqrt{5}}_Q$ and in  $\mathcal{G}^{>  2+\sqrt{5}}_Q \cap  \mathcal{G}^{\leqslant  2+\varepsilon}_Q$
Their work was brought to completion two years later by Belardo et al. in \cite{BW-Conj}. Finally, Wang et al. \cite{wangsl1} gave the structure of graphs in $\mathcal{G}^{>  2+\varepsilon}_Q \cap  \mathcal{G}^{\leqslant  4.5}_Q$.

\begin{thm}\label{Qindex4} Let $\tau_1=2+\sqrt{5}$ and $\tau_2=2+\varepsilon$.
The following equalities of sets hold. 
\vspace{.6em}
\begin{itemize}
\item[$\mathrm{(i)}$]{\rm \cite{cve1}}
$\mathcal{G}^{< 4}_Q= \{ P_n \mid n \in \N \}$.\\[-.65em]
\item[$\mathrm{(ii)}$]{\rm \cite{cve1}}
$\mathcal{G}^{4}_Q= \{ K_{1,3}, C_n \mid n \geqslant 3 \}$;\\[-.65em]
\item[$\mathrm{(iii)}$]{\rm \cite{wangsl1}}
$\mathcal{G}^{> 4}_Q \cap \mathcal{G}^{\leqslant  \tau_1}_Q = \{T_{1,1,n-3} \mid n \geq 5\}$;\\[-.65em]
\item[$\mathrm{(iv)}$]{\rm \cite{wangsl1,BW-Conj}}
$\mathcal{G}^{> \tau_1}_Q \cap \mathcal{G}^{\leqslant  \tau_2}_Q = \{ T_{1,b,c} \mid c \geq b \geq 2\} \cup \{Q_{a,b,c} \mid b \geq a+c+1\}$.
\end{itemize}

\vspace{.8em} \noindent
Moreover, the set $\mathcal{G}^{> \tau_2}_Q \cap \mathcal{G}^{\leqslant  4.5}_Q$
just contains open and closed quipus.
\end{thm}

\section{Hermitian adjacency matrix}

Liu and Li \cite{liu-li} and Guo and Mohar \cite{mohar3} independently introduced the Hermitian adjacency matrix $\mathcal{H}(G)$ associated to  a mixed graph $G$.

We recall that a mixed graph $G$ consists of a vertex set $V=V(G)$ and an arc set $$\vv{E}(G) \subseteq (V \times V)\setminus \{(v,v) \mid v \in V\}.$$ Note that if an arc $uv=(u,v)$ belongs to  $\vv{E}(G)$, then the arc $vu$ may or may not belong to $\vv{E}(G)$. A {\em digon}  $\{u,v\}$ is determined by every pair of arcs $uv$ and $vu$ both belonging to $\vv{E}(G)$, and can be regarded as an undirected edge connecting $u$ and $v$.
Mixed graphs are also called {\it directed graphs} or {\it digraphs} in \cite{guo-mohar-dm}. The underlying graph of a digraph $G$ is the graph with vertex-set $V$ and edge-set
$E = \{ \{x, y\} \mid xy \in  \vv{E}(G) \; \text{or} \; yx \in \vv{E}(G) \}$.

The entries $\mathcal{H}_{uv}$ of the Hermitian adjacency matrix $\mathcal{H}(G) \in \mathbb{C}^{|V|\times|V|}$  are as follows:
$$\mathcal{H}_{uv} =
\begin{cases}
\phantom{-}1 \quad \text{if $uv$ and $vu \in \vv{E}(G)$};\\
\phantom{-}i  \quad \text{if $uv\in \vv{E}(G)$ and $vu \not\in \vv{E}(G)$};\\
-i \quad \text{if $uv \not\in \vv{E}(G)$ and $vu \in \vv{E}(G)$};\\
\phantom{-}0   \quad \text{otherwise}.
\end{cases}
$$
Spectral properties of the $\mathcal{H}$-matrix have been investigated in \cite{chen1,chen2,mohar2,guo-mohar-dm,hu-li-liu,mohar1,tian,wang-yuan1,wang-yuan,wis-dam}. 
Let $uv$ an arc in $\vv{E}(G)$ such that $vu \not\in \vv{E}(G)$. `Reversing the direction (or the orientation) of an arc $uv$' means to consider the mixed graph $G'$ obtained from $G$ by replacing $uv$ with $vu$ in its arc set. The $\mathcal{H}$-spectrum is preserved if we reverse the direction of all arcs not involved in a digon. The mixed graph obtained in this way is called the {\it converse} of the original one. However, Guo and Mohar \cite{mohar3,mohar1} unveiled a more complicated transformation leaving the $\mathcal{H}$-spectrum unchanged. Suppose that the vertex-set of $G$ is partitioned in four (possibly empty) sets, $V(G) =V_1 \cup V_{-1} \cup V_i \cup V_{-i}$. An arc $xy \in \vv{E}(G)$ is said to be of type $(j, k)$ for $j, k \in \{\pm1, \pm i\}$ if $x \in V_j$ and $y \in V_k$. The partition is {\it admissible} if the following conditions hold:
\begin{itemize}
\item[(a)]
There are no digons of types $(1, -1)$ or $(i, -i)$.\\[-.65em]
\item[(b)]
All edges of types $(1, i)$, $(i, -1)$, $(-1, -i)$, $(-i, 1)$ are contained in digons.\\[-.65em]
\end{itemize}
A {\it four-way switching} with respect to a partition $V(G) =V_1 \cup V_{-1} \cup V_i \cup V_{-i}$ is the operation of changing $G$ into the mixed graph $G'$ by making the following changes:\\[-.65em]
\begin{itemize}
\item[(i)]
reversing the direction of all arcs of types $(1, -1), (-1, 1), (i, -i), (-i, i)$;\\[-.65em]
\item[(ii)]
(replacing each digon of type $(1, i)$ with a single arc directed from $V_1$ to $V_i$ and replacing each digon of type $(-1, -i)$ with a single arc directed from $V_{-1}$ to $V_{-i}$;\\[-.65em]
\item[(iii)]
replacing each digon of type $(1, -i)$ with a single arc directed from $V_{-i}$ to $V_1$ and replacing each digon of type $(-1, i)$ with a single arc directed from $V_i$ to $V_{-1}$;\\[-.65em]
\item[(iv)]
replacing each non-digon of type $(1, -i), (-1, i), (i, 1)$ or $(-i, -1)$ with the digon.\\[-.65em]
\end{itemize}
Two mixed graphs $G_1$ and $G_2$  are {\it switching equivalent} if,  after choosing an admissible partition on their common vertex set, one of them  can be obtained from the other by a suitable sequence of four-way switchings and the operation of taking the converse. It turns out that two switching equivalent mixed graphs have the same $\mathcal{H}$-spectrum. Guo and Mohar \cite{guo-mohar-dm} make use of switching equivalence to list  all connected digraphs with $\mathcal{H}$-index less than 2.

Let $F$ be any forest. We regard $F$ as the mixed graph obtained by replacing each of its edges by a digon.  From this perspective $\mathcal{H}(F)=A(F)$, and $\rho_{\mathcal H}(T)=\rho_A(T)$ for any tree $T$.  All mixed graphs whose underlying graph is a forest $F$ are switching equivalent \cite{mohar3, liu-li}; therefore, their $\mathcal{H}$-spectra are all equal to the $A$-spectrum of $F$.

To describe Guo and Mohar's results, we  still need some extra notation. We denote by $D_n$ the directed cycle on $n$ vertices. The digraph $\widetilde{C_n}$ is obtained from $D_n$ by reversing the direction of one of the directed edges. The digraph $\widetilde{C_n'}$ is the digraph obtained from $D_n$ by replacing one edge with a digon. The digraph $\widetilde{C_n''}$ is the digraph obtained from $D_n$ by taking two consecutive arcs and then replacing the first one by a digon and reversing the direction of the second one.  A {\it quadrangle} is a mixed graph whose underlying graph is $C_4$. A quadrangle is positive if either of the following holds: it has four digons, or it has two digons and the two non-digon arcs are oriented differently with respect to the order on the
cycle $C_4$ (one clockwise and one anticlockwise), or it has no digons and two pairs of oppositely oriented arcs, or it has no
digons and all arcs are oriented in the same direction with respect to the order on the cycle. It is a negative quadrangle if it
has an even number of digons and does not fall into the three cases of positive quadrangles.

Let $a, b, c, d$ be nonnegative integers. Let $\Box_{abcd}$ be a digraph obtained from a  negative quadrangle with (consecutive) vertices $v_1, v_2, v_3, v_4$ by adding directed paths of lengths $a, b, c, d$ that are attached to $v_1,v_2,v_3,v_4$ respectively. For the mixed graphs in the statement of Theorem~\ref{tgufo} and not defined above, see Fig.~3, where each arrow from a vertex $u$ to $v$ in $V(G)$ means that the arc $uv$ belong to $\vv{E}(G)$.

\begin{figure}[h]
\begin{center}
\includegraphics[width=1\textwidth]{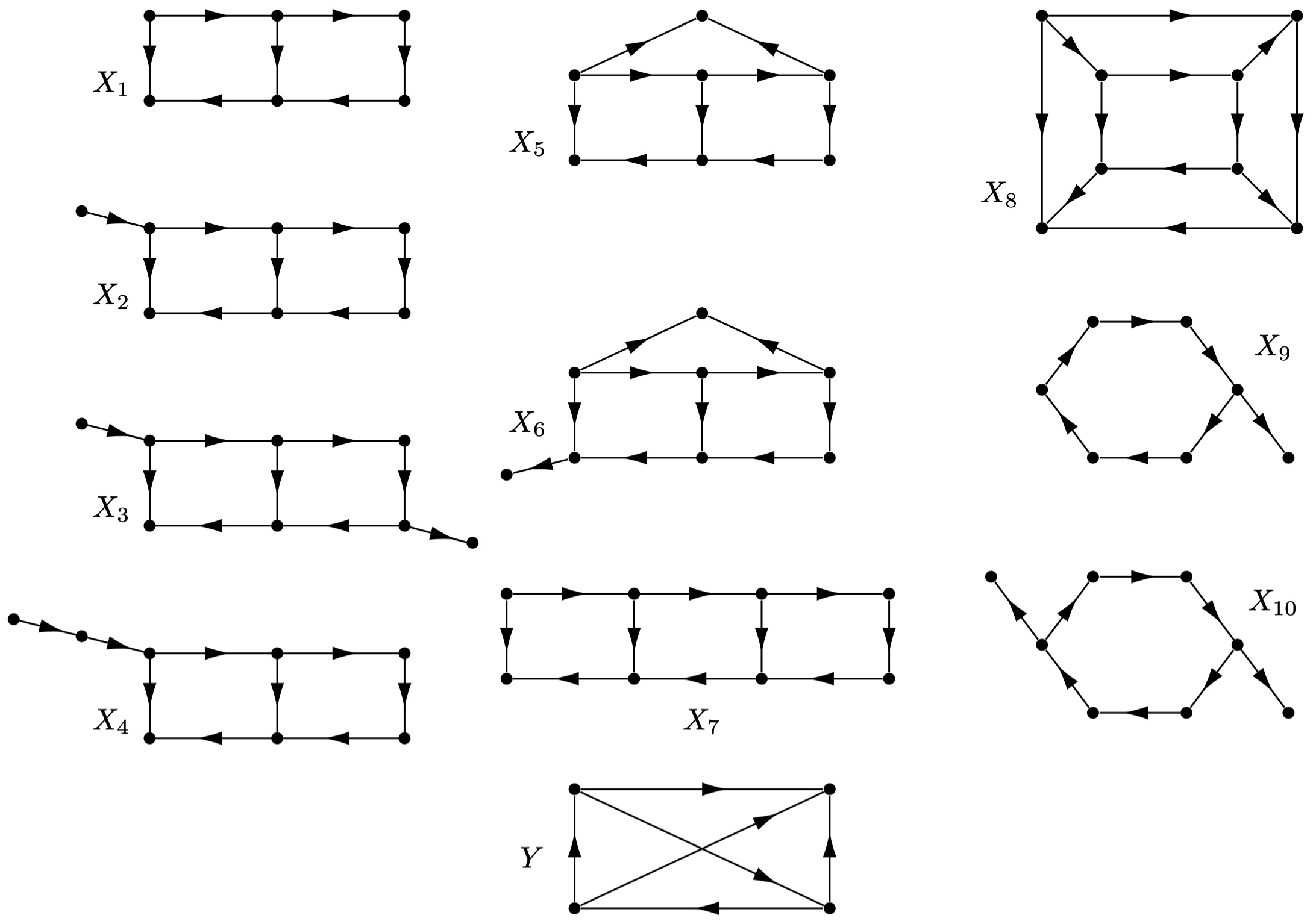}
\caption{\small The mixed graphs $X_i$ for $1 \leqslant i \leqslant 10$ and $Y$.}
\label{ex1}
\end{center}
\end{figure}

\begin{thm}{\rm \cite{guo-mohar-dm}}\label{tgufo}
A connected digraph $G$ has $\rho_{\mathcal{H}}(G) < 2$ if and only if $G$ is switching equivalent to one of the following:

{\small   \begin{multicols}{2}
   \begin{itemize}
        \item[{\rm (a)}] $P_n$;\\[-.65em]
        \item[{\rm (b)}]  $D_n$ for $n \not\equiv 0 \mod 4$;\\[-.65em]
        \item[{\rm(c)}]
 $\widetilde{C_n}$ for $n \not\equiv 2 \mod 4$;\\[-.65em]
        \item[{\rm(d)}] $\widetilde{C_n'}$ $n \not\equiv 3 \mod 4$;\\[-.65em]
        \item[{\rm(e)}]
$\widetilde{C_n''}$ $n \not\equiv 1 \mod 4$;\\[-.65em]
\item[{\rm(f)}]  $T_{a,1,1}$, with $a \in \N$;\\[-.65em]
\item[{\rm(g)}]  $T_{a,2,1}$, with $2 \leqslant a \leqslant 4$;\\[-.65em]
        \item[{\rm(h)}] $X_i$ for $ 1 \leqslant i \leqslant 10$;\\[-.65em]
         \item[{\rm(i)}] $Y$;\\[-.65em]  
         \item[{\rm(j)}]
$\square_{a,0,c,0}$, where $a \geq c \geq 0$;\\[-.65em]
\item[{\rm(k)}]
$\square_{3,1,0,0},\square_{2,1,1,0},\square_{2,1,0,0},$\\
$
\square_{1,1,1,1},\square_{1,1,1,0},\square_{1,1,0,0}$;\\[-.65em]
 \item[{\rm(l)}]
the digraph obtained from the directed triangle $D_3$ by adding a vertex and an arc from this vertex to one of the vertices of $D_3$;\\

    \end{itemize}
    \end{multicols}
    \vspace{-1em}
}
\end{thm}

The following corollary follows from the above theorem.

\begin{cor}
The smallest limit point for the $\mathcal{H}$-spectral radius of mixed graphs is $2$.
\end{cor}
As in previous sections, let $\rho_1= \sqrt{2+\sqrt{5}}$.
By Theorem~\ref{22+5} we see that there exist infinite families of trees (and hence of mixed graphs) in ${\mathcal G}_A^2$ and in   ${\mathcal G}_A^{>2} \cap {\mathcal G}_A^{< \rho_1}$. Thereby, it makes sense to  consider the following problem.
\begin{prob}
For the Hoffman program of mixed graphs with respect to the $\mathcal{H}$-matrix,
\begin{itemize}
\item[{\rm (i)}]
determine all the $\mathcal{H}$-limit points of the $\mathcal{H}$-spectral radius of mixed graphs less than $\sqrt{2+\sqrt{5}}$;
\item[{\rm (ii)}]
characterize the graphs with $\rho_{\mathcal{H}}(G) =2$ and $\rho_{\mathcal{H}}(G) \in (2, \sqrt{2+\sqrt{5}}]$.
\end{itemize}
\end{prob}

\section{Signed-adjacency matrix}

A {\it signed graph} $\Gamma = (G; \sigma)$ is a  non-empty graph $G = (V, E)$, with vertex set $V$ and edge set $E$, together with a function $\sigma: E \rightarrow \{+1, -1\}$ assigning a positive or negative sign to each edge. The (unsigned) graph $G$ is said to be the {\it underlying graph} of $\Gamma$, and the function $\sigma$ is called the {\it signature} of $\Gamma$. Unsigned graphs are treated as  signed graphs equipped with the  {\it all-positive signature $\sigma^+$} such that $\sigma^+(E)=\{1\}$.  Clearly, the {\it all-negative} signature $\sigma^-=-\sigma^+$ maps all edges onto $-1$.

For a subset $U \subseteq V(G)$, let $\Gamma^U$ be the signed graph obtained from $\Gamma$ by reversing the signs of the edges in the cut $[U, V(G) \backslash U]$, namely $\sigma_{\Gamma^U}(e) = -\sigma_{\Gamma}(e)$ for any edge $e$ between $U$ and $VG) \backslash U$, and $\sigma_{\Gamma^U}(e) = \sigma_{\Gamma}(e)$ otherwise. The signed graph $\Gamma$ and $\Gamma^U$ (and the signatures  $\sigma_{\Gamma}$ and $\sigma_{\Gamma^U}$ as well)  are said to be {\it switching equivalent}.

The {signed adjacency matrix} $\mathcal{S}(\Gamma) = (s_{ij})$ is the symmetric of $\{0,1,-1\}$-matrix such that $s_{ij} = \sigma(ij)$ whenever the vertices $i$ and $j$ are adjacent, and $a_{ij} = 0$ otherwise.  The above switching can also be explained from a matrix viewpoint. In fact, let $\Gamma$ and $\Gamma^U$ be two switching equivalent graphs. Consider the {\it signature matrix} $S_U = {\rm diag}(\epsilon_1, \epsilon_2, \cdots, \epsilon_n)$ such that
$$\epsilon_i =
\begin{cases}
+1 &\text{if $i \in U$};\\
-1, &\text{if $i \in \Gamma \backslash U$}.
\end{cases}
$$
It is easy to check that $A(\Gamma^U) = S_UA(\Gamma)S_U$. In other words, signed graphs from the same switching class share similar graph matrices by means
of signature matrices. This in particular implies that $\Gamma$ and $\Gamma^U$ are $\mathcal{S}$-cospectral. It is worthy to note that all signatures on forests are switching equivalent; moreover,
the $\mathcal S$-spectral radius is not always equal to the largest $\mathcal S$-eigenvalue, the minimal example being $(C_3, \sigma^-)$ whose $\mathcal S$-spectrum is $\{ -2, 1,1\}$.

For basic results in the theory of signed graphs, the reader is referred to \cite{zas-3,zas-4}.  On the same topic, Zaslavsky currently edits two dynamic surveys \cite{zas-1,zas-2}. For a recent list of open problems concerning signed graphs, see \cite{ber-cio-koo-wang}.



We now focus on Hoffman program in relation to signed graphs.  

\begin{prop}
The smallest limit point of $\mathcal{S}$-spectral radius of signed graphs is $2$.
\end{prop}

\begin{proof} The real number $2$ is surely an ${\mathcal S}$-limit point. In fact $\lim_{k \rightarrow \infty}\rho_{\mathcal{S}}(P_k) = 2$. In order to prove that there are no ${\mathcal S}$-limit points less than $2$, let 
 $\{\Gamma_k=(G_k, \sigma_k)\, |\, k\in \mathbb{N}\}$ be a sequence of signed graphs such that  $$\rho_{\mathcal S}(\Gamma_i) \neq \rho_{\mathcal S}(\Gamma_j) \quad \text{whenever  $i \neq j$}, \quad \text{and}  \quad \lim_{k \rightarrow \infty}\rho_{\mathcal S}(\Gamma_k) =\gamma \leqslant 2.$$
Since the $\Gamma_i$'s are pairwise distinct, and their maximum vertex degree is bounded being $\sqrt{\Delta(G_i)} \leqslant \rho_{\mathcal S}(\Gamma_i)$, then $\lim_{k \rightarrow \infty}\diam (\Gamma_k)=+\infty$. It follows that there exists a subsequence
of signed graphs  $\{\Gamma_{k_n}\, |\, n\in \mathbb{N}\}$ such that $(P_n, \sigma_{|P_n})$ is a subgraph of $ \Gamma_{k_n}$. 
Now,  $(P_n,(\sigma_{k_n})_{|P_n})$ and $P_n$ are switching equivalent, and, 
as a consequence of Cauchy Interlacing Theorem holding for all Hermitian matrices (see, for instance \cite[Theorem 0.10]{cve-book}), we have
 $\rho_{\mathcal S} (P_n, (\sigma_{k_n})_{|P_n}) \leqslant  \rho_{\mathcal S} (\Gamma_{k_n})$. Hence,
$$ 2=  \lim_{n \rightarrow \infty} \rho_{A} (P_n) =   \lim_{n \rightarrow \infty}   \rho_{\mathcal S}(P_n, (\sigma_{k_n})_{|P_n}) \leq \lim_{k \rightarrow \infty}\rho_{\mathcal S}(\Gamma_k)  =\gamma \leq 2,$$
which is possible only if $\gamma=2$.
\end{proof}

A (signed) graph is said to be {\it maximal} with respect to some property $\mathcal P$ if it is not a  proper induced subgraph of some other (signed) graph satisfying $\mathcal P$. All maximal signed graphs in $\mathcal G_{\mathcal S}^{\leqslant 2}$ have been detected by McKee and Smyth \cite{mck-smy}; they are the signed graphs $T_{2k}$ ($k \geqslant 3$), $S_{14}$ and $S_{16}$ depicted in Fig.~\ref{maximal}.

 \begin{thm}{\rm \cite{mck-smy}}\label{sign-2}
	Signed graphs in $\mathcal G_{\mathcal S}^{\leqslant 2}$ are switching equivalent to the induced subgraphs of (i) the $ 2k $-vertex toral tessellation $ T_{2k} $, for $k \geqslant 3$; (ii) the $14$-vertex signed graph $ S_{14} $; (iii) the $16$-vertex signed hypercube $ S_{16} $. Moreover, $\rho_{\mathcal S} (S_{14}) = \rho_{\mathcal S} (S_{16}) = \rho_{\mathcal S} (T_{2k}) = 2$ for all $k\geqslant 3$.
\end{thm}

\begin{figure}[h]
\begin{center}
\includegraphics[width=0.7\textwidth]{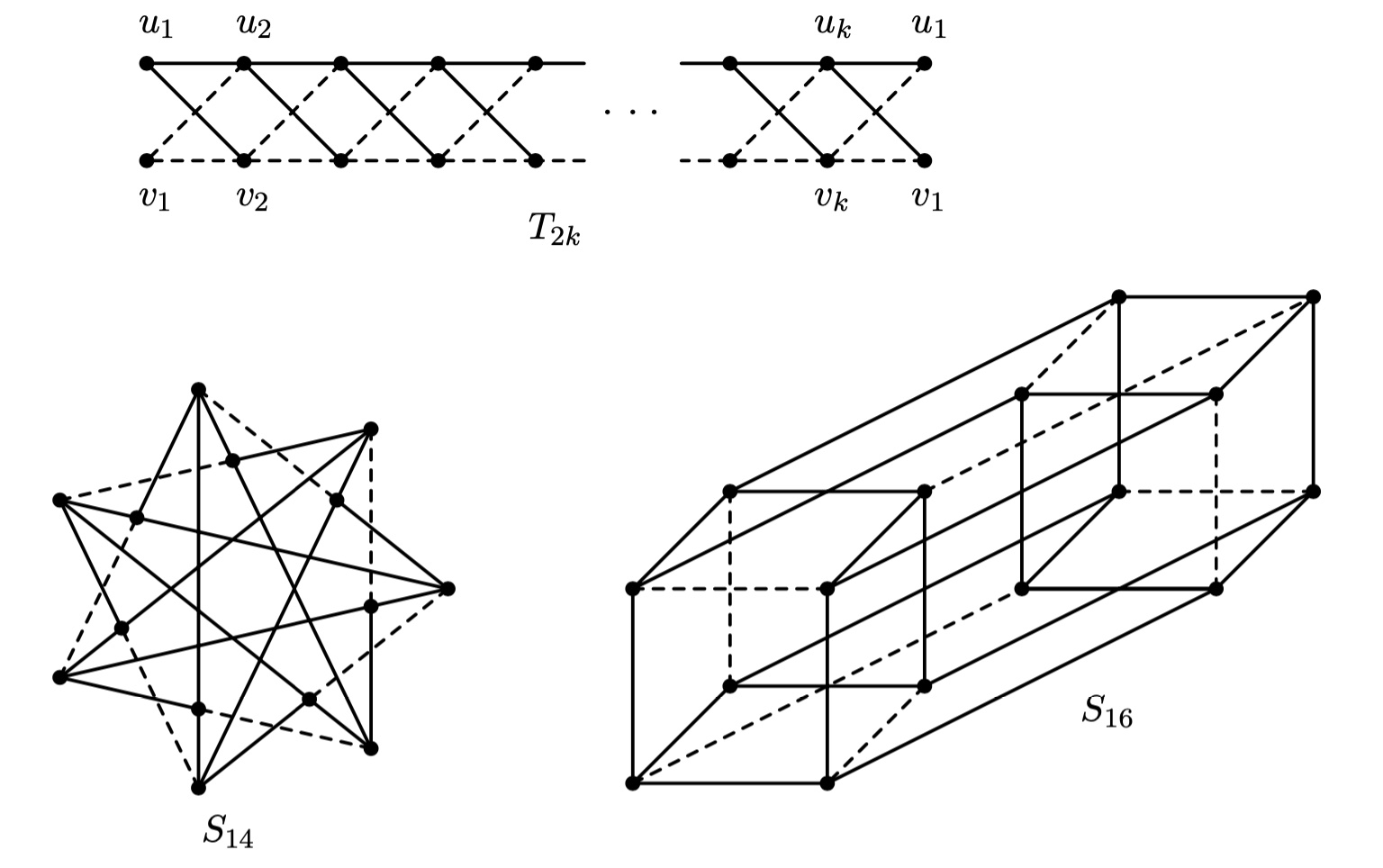}
\caption{\small Maximal signed graphs in $\mathcal G_{\mathcal S}^{\leqslant 2}$. Negative edges are depicted by dashed lines.}
\label{maximal} 
\end{center}
\end{figure}

As proved, for instance,  in \cite[Theorem 2.5]{ber-cio-koo-wang}, for a signed graph $\Gamma = (G; \sigma)$ we obtain $\rho_{\mathcal{S}}(\Gamma) \leq \rho_A(G)$. Thus, the $A$-spectral radius of the underlying graph naturally limits the magnitude of the eigenvalues of the corresponding signed graph. 

\begin{prob}{\rm \cite[Problem 3.11]{ber-cio-koo-wang}} Let $\rho_1= \sqrt{2+\sqrt{5}}$. 
Characterize all signed graphs in $\mathcal G_{\mathcal S}^{\leqslant \rho_1}$.
\end{prob}

Fortunately, the theory of limit points for the $\mathcal S$-spectral radius of signed graphs partially overlaps the one related to the  $A$-index of unsigned graphs. For instance, since all signed graphs sharing a fixed forest $F$  as underlying graph are $\mathcal S$-cospectral,
then  $\rho_{\mathcal S} (F, \sigma) = \rho_{\mathcal S} (F, \sigma^+) = \rho_A(F)$. It follows by Cauchy's Interlacing Theorem that an acyclic subgraph of a graph in $\mathcal G_{\mathcal S}^{\leqslant \rho_1}$ necessarily appears among the ones listed in Theorem~\ref{22+5}.

For the same reason, we can use the very same sequences of (acyclic) open caterpillars used by  Shearer in its proof of Theorem~\ref{she-lim}, to prove the following proposition concerning $\mathcal{S}$-spectra (an open caterpillar is a graph such that the removal of all pendant vertices results in a chordless path).

\begin{prop}
For any $\lambda \geq \sqrt{2+\sqrt{5}}$, there exists a sequence of signed caterpillars $\{\Gamma_k\, |\, k\in \mathbb{N}\}$ such that $\lim\limits_{k \rightarrow \infty}\rho_{\mathcal{S}}(\Gamma_k) = \lambda$. \end{prop}

Although Hoffman's theorem was ultimately based on a tree, its proof cannot be directly translated to $\mathcal{S}$-spectra. In fact, if $G$ is not a tree or a cycle, then $\rho_A(G) > (\sqrt{5}+1)/2$, whereas $\mathcal G_{\mathcal S}^{<2}$ contains signed unicyclic graphs  which are not cycles and signed bicyclic graphs as well (see \cite{akb-fb-eatl, BBC2}.

\begin{prob}
Characterize the limit points of the  $\mathcal{S}$-spectral radius of signed graphs less than $\sqrt{2+\sqrt{5}}$.
\end{prob}

\section{Skew-adjacency matrix}

To our knowledge, the first attempts to build a spectral theory based on skew-adjacency matrices associated to oriented graphs go back to around 2009 \cite{adi-bal-so,ima,sha-so}. The paper \cite{cave} by Cavers, Cioab\u{a} et al. provides a comprehensive introduction to this topic. An  oriented graph is a mixed graph without digons. In any case, our notation and  terminology  will be largely consistent with  \cite{stan1}.

Let  $G =(V, E)$ be an undirected  non-empty graph of order $n$, an {\it oriented graph} is a pair $\widetilde{G} = (G, \tilde{\sigma})$, where the {\it edge orientation} $\tilde{\sigma} : E \rightarrow V$ is a map satisfying $\tilde{\sigma}(ij) \in \{i, j\}$, for every $ij \in E$. As in the context of signed graphs, we say that $G$ is the {\it underlying graph} of $\widetilde{G}$. The skew-adjacency matrix $S(\widetilde{G}) = (\tilde{s}_{ij})$ of $\widetilde{G}$ is the $n \times n$ matrix defined by
$$\tilde{s}_{ij} =
\begin{cases}
\phantom{-}0 &\text{if $ij \not\in E$};\\
\phantom{-}1 &\text{if $\sigma'(ij) = i$};\\
-1 &\text{if $\sigma'(ij) = j$}.
\end{cases}
$$
Note that the non-zero eigenvalues in ${\rm Sp}_S(\widetilde{G})$ are all purely imaginary, the matrix $S(\widetilde{G})$ being real skew symmetric.  Then {\it $S$-index} $\rho_S(G')$ of $G'$ is defined as the largest modulus of the $S$-eigenvalues of $G'$. 

As in \cite{stan1, xu-gong1, xu-gong2}, if $\sigma (ij)=j$, we say that the edge $ij$ is oriented from $i$ to $j$ and write $i \!\rightarrow\!j$. Other authors  adopt the other possible choice (see for instance, \cite{cave}); in any case, these two approaches are equivalent from a spectral perspective.

Clearly, the {\it $S$-spectral radius} $\rho_S(\widetilde{G})$ of $\widetilde{G}$ is given by the largest modulus of its $S$-eigenvalues.  For any $U \subseteq V(G)$, let $\widetilde{G}^{U}$ be the oriented  graph obtained from $\widetilde{G}$ by reversing the orientation of each edge between a vertex in $U$
and a vertex in $V(G) \backslash U$. We say that $\widetilde{G}$ and $\widetilde{G}^{U}$ are {\it switching equivalent}. Note that ${\rm Sp}_S(\widetilde{G}^U)={\rm Sp}_S(\widetilde{G})$; in fact, $S(\widetilde{G}^U)$ and $S(\widetilde{G}^U)$ are similar via the matrix $S_U$  defined in Section~6.

Let $\widetilde{G}$ be an oriented graph with vertex set
$\{u_1,u_2,\dots,u_n\}$. Stani\'c \cite{stan1} defined the {\it bipartite double} ${\rm bd}(\widetilde{G})$ of $\widetilde{G}$ to be an oriented graph with vertices $\{u_11,u_12,u_21,u_22,\dots,u_n1,u_n2\}$ and  $u_ik \rightarrow u_jl$ if and only if $u_i \rightarrow u_j$ and $k \neq l$. It is easily seen that $S({\rm bd} (G'))$ is the Kronecker product $S(G') \otimes A(K_2)$.

We say that an oriented graph $\widetilde{G} = (G, \tilde{\sigma})$ is bipartite if so is $G$. The bipartite double ${\rm bd}(\widetilde{G})$ is bipartite and turns out to be connected if and only if $G$ is non-bipartite. Recall that, if $G$ is bipartite, then ${\rm Sp}_A(G)$ and ${\rm Sp}_{\mathcal S} (G, \sigma)$ are symmetric with respect to $0$ for each signature $\sigma$. Shader and So \cite{sha-so} proved that $G$ is bipartite if and only if there is an orientation $\tilde{\sigma}$ such that ${\rm Sp}_S(G,\tilde{\sigma}) = i{\rm Sp}_A(G)$.

The problem of determining the oriented graphs in $\mathcal G_S^{\leqslant 2}$ has been first investigated by Xu and Gong, and some partial results are given in \cite{xu-gong1,xu-gong2}. Stani\'c \cite{stan1}  succeeded in detecting all maximal oriented graphs in $\mathcal G_S^{\leqslant 2}$ by  forging a nice bridge between the $S$-eigenvalues of oriented graphs and the $\mathcal S$-spectrum of suitably associated signed graphs. 
Let $G=(V,E)$ be a non empty graph. A signature $\sigma: E \rightarrow \{-1,1\}$ is said to be associated to an edge orientation $\tilde{\sigma} : V \rightarrow E$ if
\begin{equation}\label{ss}
\sigma(ik)\sigma(jk) = \tilde{s}_{ik}\tilde{s}_{jk}\; \text{holds for every pair of adjacent edges}\; ik\; \text{and}\; jk.
\end{equation}
Together with the bipartite double Stani\'c provided the following two theorems  (in their statements the exponential notation is used to denote the multiplicity of an eigenvalue).

\begin{thm}{\rm \cite{stan2}}\label{ori-sign1} Let $\widetilde{G}=(G, \tilde{\sigma})$  be a bipartite oriented graph. If ${\rm rank}(S(\widetilde{G})) =2k$ and $\sigma$ is associated with $\tilde{\sigma}$, then
\[ \scalemath{.90}{ {\rm Sp}_S(\widetilde{G}) = \{ \pm i\lambda_1, \pm i\lambda_2, \ldots, \pm i\lambda_k, 0^{n-2k} \}
\quad \Longleftrightarrow \quad {\rm Sp}_{\mathcal S}(G,\sigma) = \{ \pm\lambda_1, \pm \lambda_2, \ldots, \pm \lambda_k, 0^{n-2k} \}.} \]
\end{thm}

\begin{thm}{\rm \cite{stan2}}\label{ori-sign2} Let $\widetilde{H}=(H,\tilde{\sigma})$ denote the bipartite double of the oriented graph $\widetilde{G}$, and let $\sigma$ be  the signature on $H$ associated with $\tilde{\sigma}$. If ${\rm rank}(S(\widetilde{G})) =2k$, then 
\[ \scalemath{.90}{{\rm Sp}_S(\widetilde{G}) = \{ \pm i\lambda_1, \pm i\lambda_2, \ldots, \pm i\lambda_k, 0^{n-2k} \}
\quad \Longleftrightarrow \quad {\rm Sp}_{\mathcal S}(H,\sigma) = \{ (\pm\lambda_1)^2, (\pm \lambda_2)^2, \ldots, (\pm \lambda_k)^2, 0^{2(n-2k)} \}.}\]
\end{thm}

\begin{figure}[h]
\begin{center}
\includegraphics[width=0.7\textwidth]{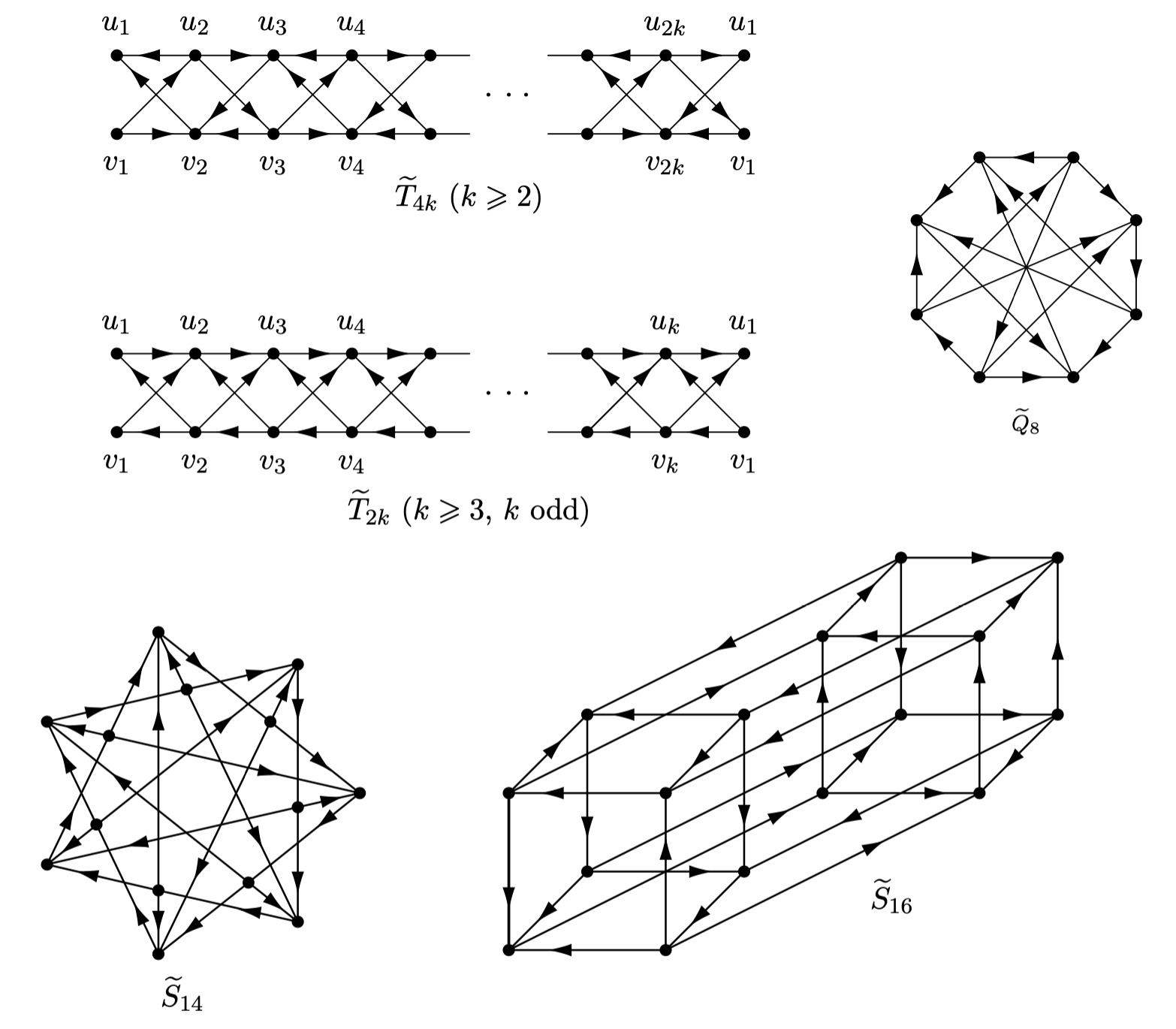}
\caption{\small Maximal connected oriented graphs whose skew spectral radius does not exceed $2$.}
\label{maximal2} 
\end{center}
\end{figure}

Theorems \ref{ori-sign1} and \ref{ori-sign2}, together with  Theorem \ref{sign-2} and \eqref{ss} are
the key ingredients to show that
if $\widetilde{G}$ is maximal  and bipartite in ${\mathcal G}_S^{\leqslant 2}$, than it is switching equivalent to an object in the set $\mathcal B=\{ \widetilde{S}_{14}, \, \widetilde{Q}_{16}, \, \widetilde{T}_{4k} \mid k \geqslant 2 \}$ (see Fig.~\ref{maximal2}). Moreover, Stani\'c proved that if $\widetilde{G}$ is a connected oriented graph such that ${\rm bd}(\widetilde{G}) \in \mathcal B$,  then $\widetilde{G}$ is switching equivalent to either $\widetilde{Q}_{8}$ or  $\widetilde{T}_{2k}$ with $k$ odd and $k \geqslant 3$.  In particular,  ${\rm bd}(\widetilde{Q}_8) = \widetilde{Q}_{16}$, and, for any odd $k \geqslant 3$, ${\rm bd}(\widetilde{T}_{2k}) = \widetilde{T}_{4k}$.
Stani\'c's results are summarized in the following theorem.

\begin{thm}{\rm \cite{stan1}}
Every maximal connected oriented graph whose skew spectral radius does not exceed $2$ is switching equivalent to one of following oriented graphs:
$$\widetilde{Q}_{8}; \qquad  \widetilde{S}_{14}; \qquad \widetilde{Q}_{16}; \qquad \widetilde{T}_{4h+2};  \qquad \widetilde{T}_{4(h+1)} \quad (h \in \N).$$
They are all illustrated in Fig.~\ref{maximal2}.
\end{thm}

\begin{prop}
For any $\lambda \geq \sqrt{2+\sqrt{5}}$, there exists a sequence of oriented graphs (namely, oriented caterpillars) $\{\widetilde{T}_k\, |\, k\in \mathbb{N}\}$ such that $\lim\limits_{k \rightarrow \infty}\rho_{S}(\widetilde{T}_k) = \lambda$.
\end{prop}

\begin{proof}
In order to prove Theorem \ref{she-lim},  for any  $\lambda \geq \sqrt{2+\sqrt{5}}$ Shearer found a  sequence of   nested caterpillars $\{T_k\, |\, k\in \mathbb{N}\}$ such that $\lim\limits_{k \rightarrow \infty}\rho_{A}(T_k) = \lambda$, In addition, Shader and So \cite{sha-so} showed that ${\rm Sp}_{S}(G, \tilde{\sigma}) = i{\rm Sp}_A(G)$ for any orientation $\tilde{\sigma}$ if and only if $G$ is a forest. Thereby, for any oriented tree $(T,\tilde{\sigma})$, we have $\rho_S(T,\tilde{\sigma}) = \rho_A(T)$. Thus, 
whatever orientation $\tilde{\sigma}_k$ we choose on the caterpillar $T_k$, we obtain 
$\rho_S(T_i,\tilde{\sigma}_i) \not= \rho_S(T_j,\tilde{\sigma}_j)$ whenever $i\not=j$ and $\lim\limits_{k \rightarrow \infty}\rho_{S}(T'_k) = \lambda$.
\end{proof}

\section{Adjacency tensor}

Since Lim \cite{lim} and Qi \cite{qi} independently introduced the eigenvalues of tensors or hypermatrices in 2005, the spectral theory of tensors has rapidly developed.  A {\it hypergraph} $H$ is a pair $(V,E)$, where $E \subseteq \mathcal{P}(V)$. The elements of $V=V(H)$ are referred to as vertices and the elements of $E=E(H)$ are called edges. A hypergraph H is said to be {\it k-uniform} for an integer $k \geq 2$ if, for all $e \in E(H)$, $ \left| e \right|=k$. To avoid trivial cases, we assume that $E$ is non-empty.

\begin{Def}{\rm \cite{coo-dut}}
Let $H$  be an $r$-uniform hypergraph. Then the adjacency tensor of $H$ is defined as $\mathcal{A}(H) = (a_{{i_1}{i_2}\ldots{i_r}})$ $k$th order and $n$-dimensional tensor, where
$$a_{{i_1}{i_2}\ldots{i_r}} =
\left\{
\begin{array}{ll}
\frac{1}{(r-1)!} & \mbox{if $\{i_1, i_2,\ldots,i_r \} \in E(H)$;}\\
0                & \mbox{otherwise}.
\end{array}
\right.
$$
\end{Def}
It is immediately seen that the adjacency tensor of hypergraphs is symmetric and generalizes the adjacency matrix of graphs. 
Let $[n]$ denote the set $\{1, 2, \dots, n\}$. The polynomial form $f_H(x)$: $\mathbb{R}^n \rightarrow \mathbb{R}$ is defined for any vector ${\rm x}=(x_1,\ldots, x_n) \in \mathbb{R}^n$ as
\begin{equation}\label{polpol}
f_H(x) =\sum_{(i_1,i_2,\dots, i_r) \in [n]^{r}}x_{i_1}x_{i_2}\cdots x_{i_r} = r\sum_{\{i_1,i_2,\ldots,i_r\} \in E(H)}x_{i_1}x_{i_2}\cdots x_{i_r}
\end{equation}

Cooper and Dutle \cite{coo-dut} call  $\lambda \in \mathbb C$ an {\it $H$-eigenvalue}  if there is a non-zero vector $x \in \mathbb C^n$ 
satisfying
$$ \sum_{(i_2,\dots, i_r) \in [n]^{r-1}}a_{j, i_2\dots i_r}x_{i_2}\cdots x_{i_r}= \lambda x_j^{r-1}$$
for all $j \in [n]$, and prove that the {\it $\mathcal{A}$-spectral radius} of $H$, i.e.\  the largest modulus among the $H$-eigenvalues, is also equal to
\begin{equation}\label{Hspect} \rho_{\mathcal{A}}(H) = \max_{\parallel {\rm x} \parallel_r=1}f_H(x), 
\end{equation}
the maximum value assumed by \eqref{polpol} over the $r$-norm unit sphere.
Note that some authors (e.g. \cite{ lu-man}) just skip the definition of an $H$-eigenvalue and define the {\it $\mathcal{A}$-spectral radius} of $H$ to be the real number \eqref{Hspect}. For more details on the eigenvalues of tensors see \cite{lim,qi}.

Lu and Man \cite{lu-man} obtained the smallest limit point of $\mathcal{A}$-spectral radii of connected $r$-uniform hypergraphs. In fact, they proved the following result.

\begin{thm}{\rm \cite{lu-man}}
The smallest limit point of $\mathcal{A}$-spectral radii of connected $r$-uniform hypergraphs is $\sqrt[r]4$.
\end{thm}

For any real number $\rho$, we denote by $\mathcal G_{\mathcal A}^\rho$ (resp.,  $\mathcal G_{\mathcal A}^{<\rho}$) the set of uniform hypergraphs whose $\mathcal A$-spectral radius is equal to $\rho$ (resp. less than $\rho$). We also set $\rho_r= \sqrt[r]4$.

There is a rationale for the way many elements detected in $\mathcal G_{\mathcal A}^{<\rho}$ and in $\mathcal G_{\mathcal A}^{\rho}$  are denoted in \cite{lu-man} and below: they are the $r$-uniform counterpart of the Smith graphs listed in Theorem \label{otto}(i) and (ii); and the Smith graphs, in turn, are all simply-laced Dynkin diagrams.  In $\mathcal G_A^{<2}$ we find, in fact, $P_n$, $T_{1,1,n-1}$; $T_{1,2,2}$, $T_{1,2,3}$ and  $T_{1,2,4}$, which are the Dynkin diagrams usually denoted by $A_n$, $D_n$, $E_6$, $E_7$ and $E_8$ respectively. The graphs  $C_n$, $W_n$ and $K_{1,4}$, $T_{2,2,2}$, $T_{1,3,3}$ and $T_{1,2,5}$ are instead the extended Dynkin diagrams respectively known as $\tilde{A}_n$, $\tilde{D}_n$, $\tilde{E}_6$, $\tilde{E_7}$, and $\tilde{E_8}$.

\begin{figure}[h]
\begin{center}
\includegraphics[width=0.8\textwidth]{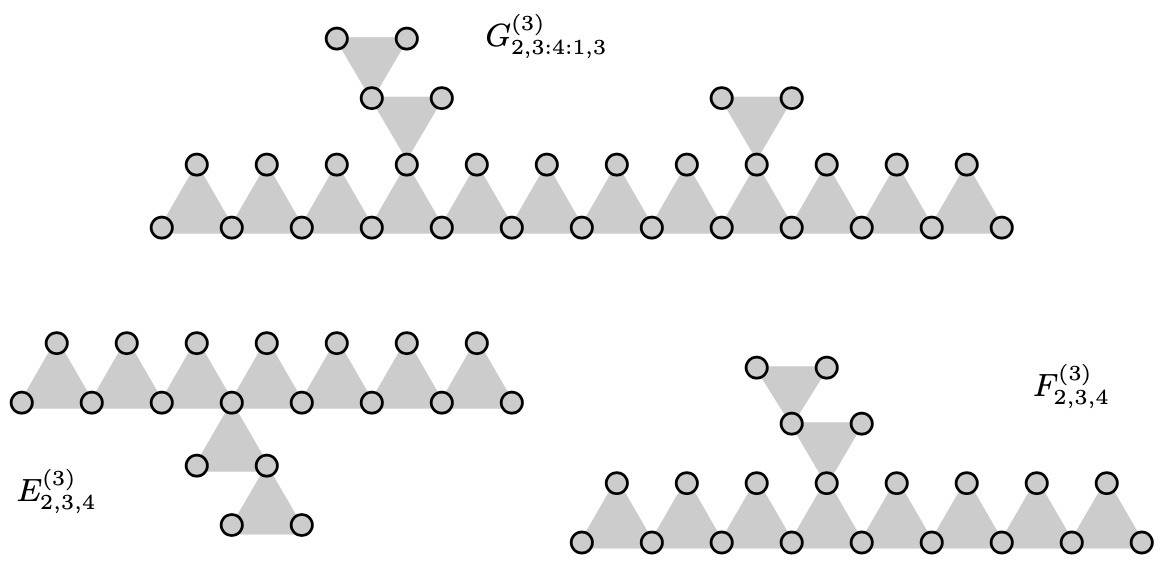}
\caption{ \label{fig400}  \small Examples of hypergraphs of type $E_{i,j,k}^{(3)}$, $F_{i,j,k}^{(3)}$ and $G^{(3)}_{i,j:k:l,t}$.}
\end{center}
\end{figure}

Apart from $C_2^{(3)}$, for all uniform hypergraphs defined in the rest of this section, we assume that adjacent edges has just one vertex in common. The $3$-uniform graphs $E^{(3)}_{i,j,k}$, $F^{(3)}_{i,j,k}$ and $G^{(3)}_{i,j:k:l,t}$   are respectively obtained: 
\begin{itemize}
\item[(E)] by attaching three hyperpaths of length $i, j, k$ to one vertex;
\item[(F)] by attaching three hyperpaths of length $i, j, k$ to each vertex of a fixed edge;
\item[(G)] by attaching four hyperpaths of length $i, j, l, t$ to four ending vertices of a hyperpath $A^{(3)}_{k+2}$ of length $k+2\geqslant 2$ (see Fig.~\ref{fig400}).
\end{itemize}
To make notation consistent with the $r=2$ case, we set: 
$E^{(3)}_6 :=E^{(3)}_{1,2,2}$, $E^{(3)}_7 :=E^{(3)}_{1,2,3}, E^{(3)}_8 :=E^{(3)}_{1,2,4}, \tilde{E}^{(3)}_6 :=E^{(3)}_{2,2,2}, \tilde{E}^{(3)}_7 :=E^{(3)}_{1,3,3}$, $\tilde{E}^{(3)}_8 :=E^{(3)}_{1,2,5}$, and $D^{(3)}_n :=E^{(3)}_{1,1,n-2}$.

We are now in the stage to describe the elements in  $\mathcal G_{\mathcal A}^{\rho_r}$ and in  $\mathcal G_{\mathcal A}^{<\rho_r}$. Lu and Man \cite{lu-man} first characterized the $3$-uniform hypergraphs in the two sets, finding the $r$-uniform hypergraphs for $r \geqslant 4$ at a later time.
\begin{thm}{\rm \cite{lu-man}}\label{lu-1}
Let $\rho_3 = \sqrt[3]4$. If a  $3$-uniform hypergraph $H$ belongs to $\mathcal G_{\mathcal A}^{<\rho_3}$, then it is one of the following graphs:
{\small   \begin{multicols}{2}
\begin{itemize}
\item[{\rm (i)}]
the hyperpath $A^{(3)}_m$ of $m\geqslant 1$ edges.\\[-.8em]
\item[{\rm (ii)}]
$D^{(3)}_m$ for $m \geqslant 3$;\\[-.8em]
\item[{\rm (iii)}]
$D'^{(3)}_m:=F^{(3)}_{1,1,m-3}$ for $m \geqslant 4$;\\[-.8em]
\item[{\rm (iv)}]
$B^{(3)}_m=F^{(3)}_{1,2,m-4}$ for $m \geqslant 5$;\\[-.8em]
\item[{\rm (v)}]
$B'^{(3)}_m:=G^{(3)}_{1,1:(m-6):1,1}$ for $m \geqslant 6$;\\[-.8em]
\item[{\rm (vi)}]
$\overline{B}^{(3)}_m:= G^{(3)}_{1,1:(n-7):1,2}$ for $m \geqslant 7$;\\[-.8em]
\item[{\rm (vii)}]
$BD^{(3)}_m$ for $m \geq 5$ (see Fig.~\ref{fig500});\\[-.8em]
\item[{\rm (viii)}]
Thirty-one exceptional 3-uniform hypergraphs: $E^{(3)}_6$; $E^{(3)}_7$; $E^{(3)}_8$; $F^{(3)}_{1,3,k}$ (for $3 \leqslant k \leqslant 13$); $F^{(3)}_{1,4,k}$ (for $4 \leqslant k \leqslant 7$);  $F^{(3)}_{1,5,5}$; $F^{(3)}_{2,2,k}$ (for $2 \leqslant k \leqslant 6$);  $F^{(3)}_{2,3,3}$;   and $G^{(3)}_{1,1:k:1,3}$ (for $0 \leqslant k \leqslant 5$).
\end{itemize}
\end{multicols} }
\end{thm}

\begin{figure}[h]
\begin{center}
\includegraphics[width=0.7\textwidth]{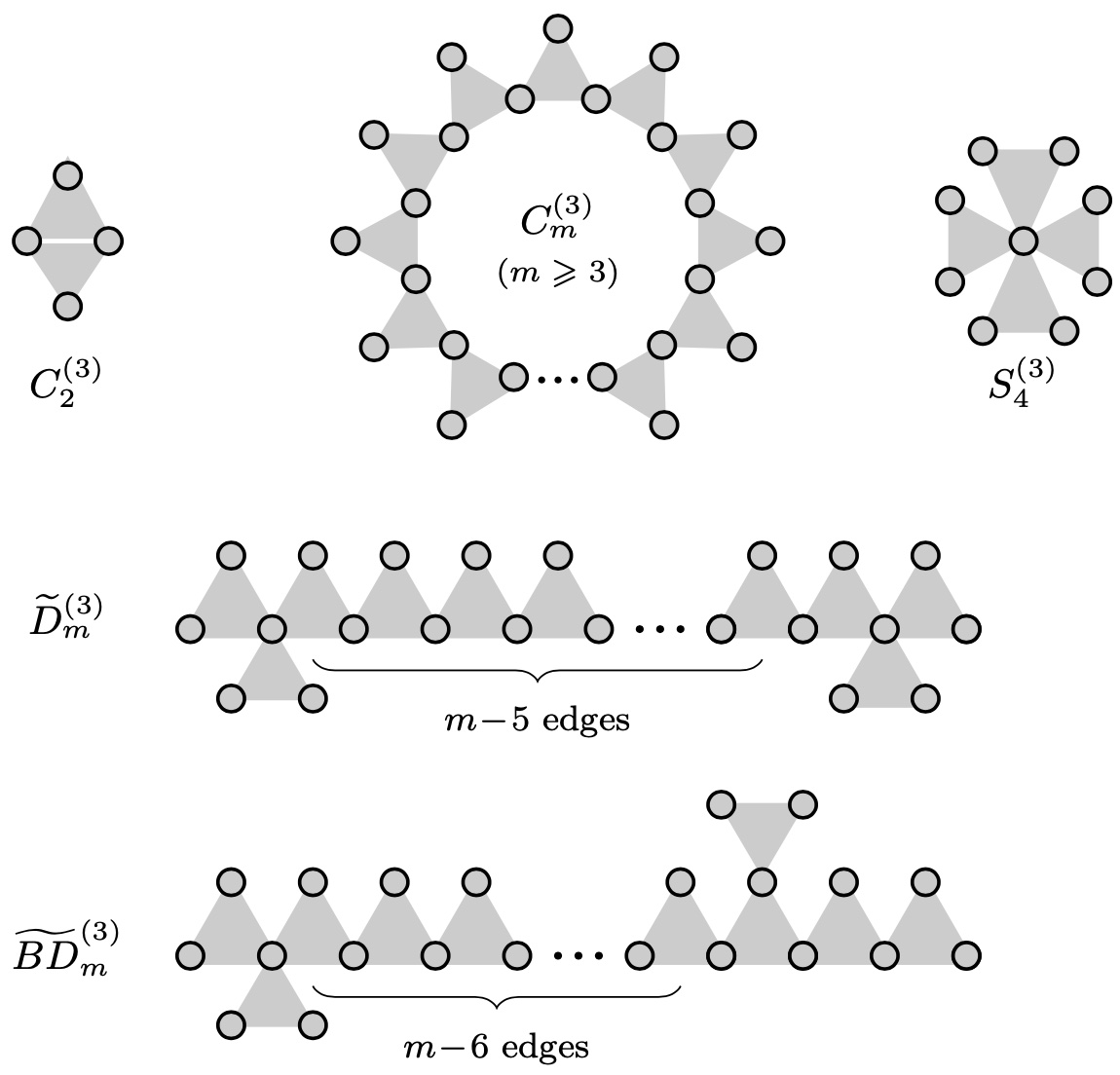}
\caption{ \label{fig500}  \small Some hypergraphs involved in Theorems~\ref{lu-1} and~\ref{lu-2}.}
\end{center}
\end{figure}

\begin{thm}{\rm \cite{lu-man}}\label{lu-2}
Let $\rho_3 = \sqrt[3]4$. If a  $3$-uniform hypergraph $H$ belongs to ${\mathcal G}_{\mathcal A}^{\rho_3}$, then $H$ is one of the following hypergraphs:
\begin{multicols}{2}
\begin{itemize}
\item[{\rm (i)}]
the hypercycle $C^{(3)}_m$ for $m \geqslant 3$;\\[-.5em]
\item[{\rm (ii)}]
$\tilde{D}^{(3)}_m$ for $m \geq 5$;\\[-.5em]
\item[{\rm (iii)}]
$\tilde{B}^{(3)}_m:=G^{(3)}_{1,2:(m-8):1,2}$ for $m \geq 8$;\\[-.5em]
\item[{\rm (iv)}]
$\widetilde{BD}^{(3)}_m$ for $m \geq 6$ (see Fig.~\ref{fig500});
\item[{\rm (v)}]
Twelve exceptional 3-uniform hypergraphs: $C^{(3)}_2, S^{(3)}_4$, $\tilde{E}^{(3)}_6, \tilde{E}^{(3)}_7$, $\tilde{E}^{(3)}_8$, $F^{(3)}_{2,3,4}$, $F^{(3)}_{2,2,7}$, $F^{(3)}_{1,5,6}$, $F^{(3)}_{1,4,8}$, $ F^{(3)}_{1,3,14}$, $G^{(3)}_{1,1:0:1,4}$, and $G^{(3)}_{1,1:6:1,3}$. (see Fig.~\ref{fig400} and~\ref{fig500}.)
\end{itemize}
\end{multicols}
\end{thm}
\medskip

A hypergraph $H=(V, E)$ is called {\it reducible} if every edge $e$ contains at least one leaf vertex $v_e$. In this case, we can define an $(r-1)$-uniform multi-hypergraph $H=(V, E)$ by removing $v_e$ from each edge $e$, i.e., $V'=V \backslash \{v_e | e \in E\}$ and $ E=\{e -v_e | e \in E\}$. We say that $H'$ is {\it reduced} from $H$, whereas $H$ {\it extends}  $H'$. As proved in \cite{lu-man}, If $H$ extends $H'$, then $H \in \mathcal G_{\mathcal A}^{\rho_r}$ (resp., $H \in \mathcal G_{\mathcal A}^{<\rho_r}$) if and only if 
$H' \in \mathcal G_{\mathcal A}^{\rho_{r-1}}$ (resp., $H' \in \mathcal G_{\mathcal A}^{<\rho_{r-1}}$).

\begin{thm}{\rm \cite{lu-man}}\label{lu-3}
Let $r \geqslant 4$ and $\rho_r= \sqrt[r]4$. If an $r$-uniform 
hypergraphs lies in $\mathcal G_{\mathcal A}^{<\rho_r}$, then it must be one of following graphs:
\begin{itemize}
\item[{\rm (i)}]
$A^{(r)}_n$, $D^{(r)}_n$, $D'^{(r)}_n$, $B^{(r)}_n$, $B'^{(r)}_n$, $\bar{B}^{(r)}_n$, $BD^{(r)}_n$, $E^{(r)}_6$, $E^{(r)}_7, E^{(r)}_8$, $F^{(r)}_{2,3,3}$, $F^{(r)}_{2,2,j}$ (for $2 \leqslant j \leqslant 6$), $F^{(r)}_{1,3,j}$ (for $3 \leqslant j \leqslant 13$), $F^{(r)}_{1,4,j}$ (for $4 \leqslant j \leqslant 7$), $F^{(r)}_{1,5,5}$, and $G^{(r)}_{1,1:j:1,3}$ (for $0 \leqslant j \leq 5$). These are the $r$-uniform hypergraphs obtained by extending $r-3$ times the hypergraphs in the list of Theorem \ref{lu-1};\\[-.8em]
\item[{\rm (ii)}]
$H^{(r)}_{1,1,1,1}$, $H^{(r)}_{1,1,1,2}$, $H^{(r)}_{1,1,1,3}$, $H^{(r)}_{1,1,1,4}$. These are the $r$-uniform hypergraphs obtained by extending $r-4$ times the hypergraphs  in Fig.~\ref{fig600}.
\end{itemize}
\end{thm}

  \begin{figure}[h]
\begin{center}
\includegraphics[width=0.8\textwidth]{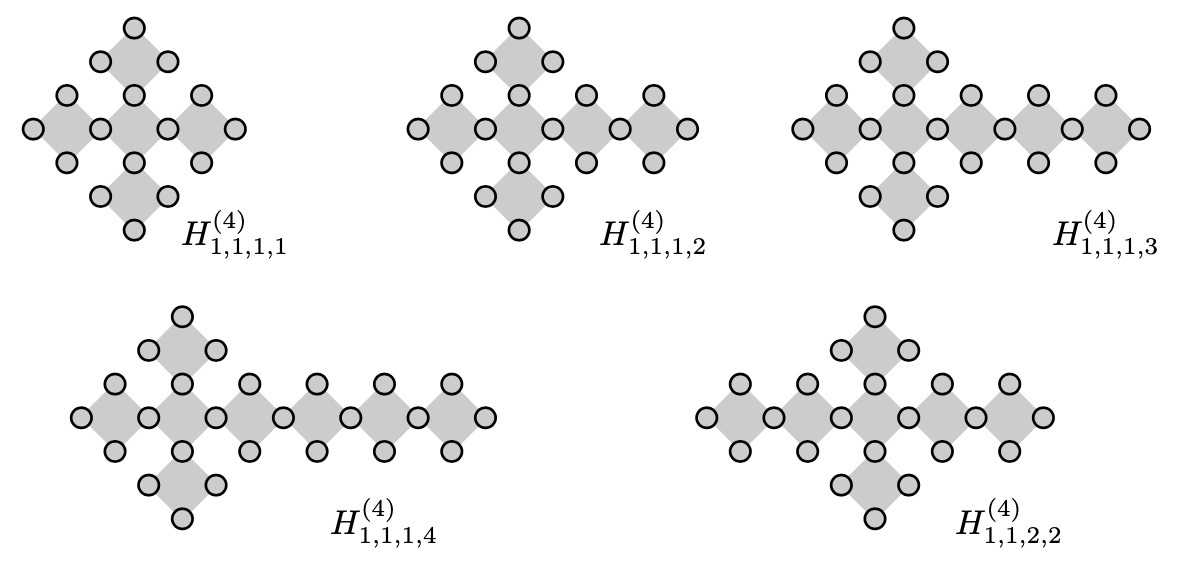}
\caption{ \label{fig600}  \small Some hypergraphs involved in Theorems~\ref{lu-3} and~\ref{lu-4}.}
\end{center}
\end{figure}

\begin{thm}{\rm \cite{lu-man}}\label{lu-4}
Let $r \geq 4$ and $\rho_r= \sqrt[r]4$. If an $r$-uniform hypergraph $H$ belongs to $\mathcal G_{\mathcal A}^{\rho_r}$, then $H$ must be one of the following graphs:
\begin{itemize}
\item[{\rm (i)}]
$C^{(r)}_n$, $\tilde{D}^{(r)}_n$, $\tilde{B}^{(r)}_n$, $\widetilde{BD}^{(r)}_n$, $C^{(r)}_2$, $S^{(r)}_4$,  $\tilde{E}^{(r)}_6$, $\tilde{E}^{(r)}_7$,  $\tilde{E}^{(r)}_8$, $F^{(r)}_{2,3,4}$, $F^{(r)}_{2,2,7}$, $F^{(r)}_{1,5,6}$, $F^{(r)}_{1,4,8}$,$F^{(r)}_{1,3,14}$, $G^{(r)}_{1,1:0:1,4}$, and $G^{(r)}_{1,1:6:1,3}$. These are the $r$-uniform hypergraphs obtained by extending $r-3$ times the hypergraphs in the list of Theorem \ref{lu-2};
\item[{\rm (ii)}]
the hypergraph $H^{(r)}_{1,1,2,2}$ obtained by extending $r-4$ times the hypergraph $H^{(4)}_{1,1,2,2}$ in Fig.~\ref{fig600}.
\end{itemize}
\end{thm}

Taken a careful look to the proofs in \cite{lu-man}, our experience suggests that the next considerable limit point for $\mathcal{A}$-spectral radius of connected $r$-uniform hypergraphs should be the number $\sqrt[r]{2+\sqrt{5}}$.

\begin{prob}
For the adjacency tensors of connected $r$-uniform hypergraphs,
\begin{itemize}
\item[{\rm (i)}]
determine the limit points of $\mathcal{A}$-spectral radius less than $\sqrt[r]{2+\sqrt{5}}$, and further identify all of them;
\item[{\rm (ii)}]
establish whether each real number exceeding $\sqrt[r]{2+\sqrt{5}}$ is an $\mathcal{A}$-limit point;
\item[{\rm (iii)}]
characterize the $r$-uniform hypergraphs whose $\mathcal{A}$-spectral radius is  at most $\sqrt[r]{2+\sqrt{5}}$.
\end{itemize}
\end{prob}

Other interesting  fields of investigation  are the {\it signless Laplacian tensor} and {\it Laplacian tensor} of the uniform hypergraph $H$, defined as  $\mathcal{Q}(H) = \mathcal{D}(H) + \mathcal{A}(H)$ and $\mathcal{L}(H) = \mathcal{D}(H) - \mathcal{A}(H)$, where
$\mathcal{D}(H)$ is the {\it diagonal tensor} of order $k$ and dimension $n$, whose diagonal entry $D_{ii \ldots i}$ is the degree of the vertex $i$ for all $i \in [n]$ (see  \cite{qi-LQ}). As far as we know, the Hoffman program with respect to these two tensors haven't yet been studied. Results concerning the Hoffman program for the (signless)  Laplacian matrices in Sections 3 and 4 bring us to pose the following problem.

\begin{prob}
For the Laplacian and the signless Laplacian tensors of connected $r$-uniform hypergraphs,
\begin{itemize}
\item[{\rm (i)}]
prove that the smallest limit point of the $\{\mathcal{L},\mathcal{Q}\}$-spectral radius is $\sqrt[r]{16}$.
\item[{\rm (ii)}]
characterize the  $r$-uniform hypergraphs with $\{\mathcal{L},\mathcal{Q}\}$-spectral radius  at most $\sqrt[r]{16}$.
\item[{\rm (iii)}]
determine the  $\{\mathcal{L},\mathcal{Q}\}$-limit points which are less than $\sqrt[r]{2+\varepsilon}$;
\item[{\rm (iv)}]
establish whether each real number exceeding $\sqrt[r]{2+\varepsilon}$ is an  $\{\mathcal{L},\mathcal{Q}\}$-limit point;
where $\varepsilon=\frac{{\;}1{\;}}{3}\left((54 - 6\sqrt{33})^{\frac{{\;}1{\;}}{3}} + (54 + 6\sqrt{33})^{\frac{{\;} 1{\;}}{3}} \right)$ like in Section 2.
\item[{\rm (v)}]
characterize the $r$-uniform hypergraphs with $\{\mathcal{L},\mathcal{Q}\}$-spectral radius at most $\sqrt[r]{2+\sqrt{5}}$.
\end{itemize}
\end{prob}

\section{$A_\alpha$-matrix}

Nikiforov \cite{niki} defined the $A_{\alpha}$-matrix of a graph $G$ to be the convex linear combination $$A_{\alpha}(G)=\alpha D(G)+(1-\alpha)A(G), \quad \alpha \in [0,1].$$ Such matrix not only merges the $A$-spectra and $Q$-spectra, but offer a different perspective to generalize and deepen the spectral properties of graphs.
Clearly, $$A(G) = A_0(G), \;\; Q(G)=2A_{1/2}(G),\;\; \text{and} \;\;  L(G)=\frac{1}{\alpha-\beta}(A_\alpha(G)-A_\beta(G)) \;\; \text{ for all $\alpha \not=\beta$}.$$
An interesting literature on Nikiforov's matrix is growing rapidly; from the subsequent papers by Nikiforov {\em et al.} \cite{niki-pas,niki-rojo} to the recent applications to signed, mixed and gain graphs \cite{BBC, LW}, there are already more than fifty published papers on the spectral properties of the $A_{\alpha}$-matrix.

A first attempt to study the limit points of the $A_{\alpha}$-spectral radius of graphs has been already performed.

\begin{thm}{\rm \cite{WWXB}}
The smallest $A_\alpha$-limit point for the $A_{\alpha}$-spectral radius of graphs is $2$.
\end{thm}

The connected graphs with $A_\alpha$-index at most $2$ are also characterized. In Theorem~\ref{Aa-graph2}, the following four numbers 
are of considerable importance:
$$ s_1=\dfrac{4}{n+1+\sqrt{(n+1)^{2}-16}};$$
the root $s_2=0.2192+$ of the polynomial $2\alpha^3-11\alpha^2+16\alpha-3$; the root $s_3=0.1206+$ of $\alpha^3-6\alpha^2+9\alpha-1$; and the root $s_4=0.0517+$ of $2\alpha^3-13\alpha^2+20\alpha-1$.

\begin{thm}{\rm \cite{WWXB}}\label{Aa-graph2}
Let $G$ be a connected graph with order $n$, and let $\alpha \in [0,1]$. The following two statements hold.  
\begin{itemize}
\item[$\mathrm{(i)}$]
$\rho_{A_\alpha}(G)<2$ if and only if $G$ is one of the following graphs:
\begin{itemize}
\item[$\mathrm{(a)}$]
$P_n$ ($n\geq 1$) for $\alpha \in [0,1)$;
\item[$\mathrm{(b)}$]
$T_{1,1,n-3}$ ($n \geq 4$) for $\alpha \in [0,s_1)$;
\item[$\mathrm{(c)}$]
$T_{1,2,2}$ for $\alpha \in [0,s_2)$, $T_{1,2,3}$ for $\alpha \in [0,s_3)$ and  $T_{1,2,4}$ for $\alpha \in [0,s_4)$.
\end{itemize}
\item[$\mathrm{(ii)}$]
$\rho_{A_\alpha}\!(G)=2$ if and only if  $G$ is one of the following graphs:
\begin{itemize}
\item[$\mathrm{(a)}$]
$C_n$, $n\geq 3$;
\item[$\mathrm{(b)}$]
$P_n$ ($n \geq 3$) for $\alpha=1$;
\item[$\mathrm{(c)}$]
$W_n\; (n \geq 6)$ for $\alpha=0$;
\item[$\mathrm{(d)}$]
$T_{1,1,n-3}$ for $\alpha=s_1$;
\item[$\mathrm{(e)}$]
$T_{1,2,2}$ for $\alpha=s_2$, $T_{1,2,3}$ for $\alpha=s_3$, $T_{1,2,4}$ for $\alpha=s_4$;
\item[$\mathrm{(f)}$]
$T_{1,3,3}$, $T_{1,2,5}$, $K_{1,4}$ and $T_{2,2,2}$, for $\alpha=0$.
\end{itemize}
\end{itemize}
\end{thm}

The Hoffman program for the adjacency and the signless Laplacian matrix suggests that a natural second step is to identify all the possible $A_{\alpha}$-limit points which are bigger than $2$. In view of this goal, it is necessary to investigate the $A_\alpha$-spectral properties of graphs in the first instance. We do this in Subsection \ref{basic}, whereas  we find in Subsection \ref{sub2} many  $A_{\alpha}$-limit points larger than $2$ related to suitably built sequences of graphs which already turned out to be useful to detect the important $L$-limit point $\omega$ (see \cite{guo-lim} and Section 3).

\subsection{Some general results on the $A_\alpha$-matrix}\label{basic}

We start by fixing some notation. Let $u$ and $v$ be two vertices of a connected graph $G$. As usual, we denote  by $N_G(u)$ the neighbourhood of $u$ in $G$, i.e.\ the set of vertices in $V(G)$ adjacent to $u$, and by $d(u,v)$ the number of edges in a shortest path connecting $u$ and $v$. Throughout this section, we write the matrices $A_{\alpha}(P_n)$ and $A_{\alpha}(C_n)$  according to vertex labellings $u_1, \dots, u_n \in V(P_n)$ and $v_1, \dots, v_n \in V(C_n)$ such that $u_i$ (resp. $v_i$) is adjacent to $u_{i-1}$ and $u_{i+1}$ (resp.,  $v_{i-1}$ and $v_{i+1}$) for $2 \leqslant i \leqslant n-1$.

\begin{lem}{\rm \cite{niki}}\label{alpha-delta} For every connected graph $G$ with maximum vertex degree $\Delta$, and for every 
$\alpha \in [0,1]$, the $A_{\alpha}$-spectral radius $\rho\!_{_{A_{\alpha}}}(G)$ satisfies the following properties:

:
\begin{itemize}
\item[$\mathrm{(i)}$]
$\frac{1}{2}\left(\alpha(\Delta+1)+\sqrt{\alpha^2(\Delta+1)^2+4\Delta(1-2\alpha)}\right) \leqslant \rho\!_{_{A_{\alpha}}}\!(G) \leqslant \Delta$;
\item[$\mathrm{(ii)}$]
if $H$ is a proper subgraph of $G$, then $\rho\!_{_{A_{\alpha}}}\!(H)<\rho\!_{_{A_{\alpha}}}\!(G)$;
\item[$\mathrm{(iii)}$] if $0 \leqslant \alpha < \beta \leqslant 1$, then
$ \rho\!_{_{A_{\alpha}}}\! (G) < \rho\!_{_{A_{\beta}}}\!(G)$.
\end{itemize}
\end{lem}

\begin{lem}{\rm \cite{niki-pas}}{\label{alpha-pn}}
The $A_{\alpha}$-spectral radius of the path $P_{n}$ satisfies the following inequalities.
\begin{itemize}
\item[$\mathrm{(i)}$]
$ \displaystyle \rho\!_{_{A_{\alpha}}}\!(P_{n})\leqslant
\begin{cases}
\small   2\alpha+2(1-\alpha) \cos (\frac{\pi}{n+1} ) &\text{for $\alpha\in[0,1/2)$,}\\[.8em]
 2\alpha+2(1-\alpha) \cos (\frac{\pi}{n}) &\text{for $\alpha \in [1/2,1]$.}
\end{cases}$
\medskip

\noindent
Equality holds if and only if $\alpha=0$, $\alpha=1/2$, $\alpha=1$.\\
\item[$\mathrm{(ii)}$]
$ \displaystyle\rho\!_{_{A_{\alpha}}}\!(P_{n})\geqslant
\begin{cases}
2\alpha+2(1-\alpha) \cos(\frac{\pi}{n}) &\text{for $\alpha\in[0,1/2)$}\\[.8em]
2\alpha+2\alpha \cos (\frac{\pi}{n})-2(2\alpha-1) \cos (\frac{\pi}{n+1}) &\text{for $\alpha\in[1/2,1]$.}\\
\end{cases}$
\medskip

\noindent
Equality holds if and only if $\alpha=1/2$.
\end{itemize}
\end{lem}
Let $\phi(G) = \det(\lambda I - A_\alpha(G))$ denote the $A_{\alpha}$-polynomial of a graph $G$. For every vertex $v \in V(G)$, we indicate by
$A_{\alpha}(G)_v$ the principal submatrix of
$A_{\alpha}(G)$ obtained by deleting the row and the column corresponding to the vertex $v$, and by
$\phi(G)_v$
the characteristic polynomial of $A_{\alpha}(G)_v$.
\begin{lem}\label{alpha-g1g2}{\rm \cite{li-chen-meng}}
The $A_{\alpha}$-characteristic polynomial of $G=G_1u \!\!:\!\!v G_2$,  the graph obtained by joining the vertex $u$ of the graph $G_1$ to the vertex $v$ of the graph $G_2$ by an edge, is given by the following formula.
$$
\phi(G)=\phi(G_1)\phi(G_2)-\alpha\phi(G_1)_u\phi(G_2)-\alpha\phi(G_1)\phi(G_2)_v+(2\alpha-1)\phi(G_1)_u\phi(G_2)_v.
$$
\end{lem}

For each positive integer $n$, we consider the matrix $B_{n}$ obtained from $A_{\alpha}(P_{n+1})$ by deleting the row and column corresponding to the end vertex $u_1$ of the path $P_{n+1}$, and the matrix $H_{n}$ obtained from
$A_{\alpha}(P_{n+2})$ by deleting the rows and the columns corresponding to both end-vertices of $P_{n+2}$. Clearly,
both $B_n$ and $H_n$ are $n \times n$ matrices. We explicitly observe that the three matrices $P_n$, $B_n$ and $H_n$ are equal if and only if $\alpha=0$.

For every $\alpha \in [0,1)$, we also set
\[  \phi(P_0)=\frac{1-2\alpha}{(1-\alpha)^2}, \quad \phi(B_0)=1, \quad 	\text{and}   \quad \phi(H_0)=1. \]

\begin{lem}\label{e0}  The equation
\begin{equation}\label{e1} \left(\lambda+\frac{1}{\alpha}-2 \right)\phi(B_n)=\phi(P_{n+1})+
\frac{(1-\alpha)^2}{\alpha}\,  \phi(P_n)
\end{equation}
holds for every $n \geqslant 0$ and $\alpha \in (0,1)$.
\end{lem}
\begin{proof} For $n\in \{0,1\}$, the statement follows from a direct calculation. For $n>1$, the matrix $B_n$ has the form
\begin{equation}\label{bn}
\left( \begin{array}{cccccc}
\lambda-2\alpha  &  \alpha-1    &   0           &    \cdots   &   0           \\
       \alpha-1   &   \lambda-2\alpha  &   \alpha-1    &    \cdots   &   0           \\
           0      &   \alpha-1   &   \lambda-2\alpha   &    \cdots   &   0           \\
       \vdots     &   \vdots     &   \vdots      &    \ddots   &   \vdots      \\
           0      &   0          &    0          &    \cdots   &   \lambda-\alpha    \\
\end{array}\right).
\end{equation}
The first row of \eqref{bn} is equal to the sum $(\lambda-\alpha, \alpha -1, 0, \dots 0) + (-\alpha, 0, \dots 0)$. By linearity of the determinant function in the first row, we get
\begin{equation}\label{e2} \phi(B_n) = \phi(P_n)-\alpha\phi(B_{n-1}).
\end{equation}
We now use Lemma \ref{alpha-g1g2}, in the case $G_1=P_n$, $u$ is an end vertex of $G_1$, and $G_2=P_1$. Clearly,
$G_1u \!:\! vG_2=P_{n+1}$  and
\begin{equation}\label{PPB}
\phi(P_{n+1})=(\lambda-\alpha)\phi(P_n)+(2\alpha-\alpha \lambda-1)\phi(B_{n-1}).
\end{equation}
Combining \eqref{e2} and \eqref{PPB}, we arrive at \eqref{e1}.
\end{proof}
\begin{lem}\label{e1234} For every $n \geqslant 1$ and $\alpha \in [0,1)$, the following equalities hold.
\begin{itemize}
\item[$\mathrm{(i)}$]
$\phi(P_{n+1})=(\lambda-2\alpha)\phi(P_n)-(1-\alpha)^2\phi(P_{n-1})$;
\item[$\mathrm{(ii)}$]
$\phi(P_{n+1})=\lambda\phi(H_{n})+(2\alpha-1)\phi(H_{n-1})$;
\item[$\mathrm{(iii)}$]
$\phi(C_{n+2})=(x-2\alpha)\phi(H_{n+1})-2(\alpha-1)^2\phi(H_{n})+2(-1)^{n+1}(\alpha-1)^{n+2}$.
\end{itemize}
\end{lem}

\begin{proof}

 For $n=1$, (i) and (ii)  follow from a direct calculation. Let now $n>1$. A cofactor expansion along the row corresponding to an end vertex of $P_{n+1}$ suffices to show that (i) holds for $\alpha=0$. For $\alpha>0$,  (i) follows from \eqref{e1} and \eqref{PPB}.

We now prove (ii) for $n>1$. By expanding the determinant $\phi(H_{n+1})$ by the first row we get
\begin{equation}\label{hn}
\phi(H_{n+1})=(\lambda-2\alpha)\phi(H_{n})-(1-\alpha)^2\phi(H_{n-1}).
\end{equation}
Once we write the first and the last row of $A_{\alpha}(P_{n+1})$ as
$(\lambda-2\alpha, \alpha-1,0, \dots 0) + (\alpha,0, \dots, 0)$ and $(0, \dots,0 , \alpha-1, \lambda-2\alpha) + (0, \dots,0,\alpha)$, linearity of the determinant function in the first and in the last row gives
\[ \phi(P_{n+1}) = \phi(H_{n+1})+2\alpha\phi(H_{n})+\alpha^2\phi(H_{n-1})
\]
which, together with \eqref{hn}, leads to (ii).\smallskip

For (iii),  we consider the cofactor expansion of $\phi(C_{n+2})$ along the first row, getting
{\footnotesize
\begin{equation*}
\begin{split}
\phi(C_{n+2})&=\left|
\begin{array}{cccccc}
       \lambda-2\alpha  &  \alpha-1    &   0           &    \cdots   &   \alpha-1           \\
       \alpha-1   &   \lambda-2\alpha  &   \alpha-1    &    \cdots   &   0           \\
           0      &   \alpha-1   &   \lambda-2\alpha   &    \cdots   &   0           \\
       \vdots     &   \vdots     &   \vdots      &    \ddots   &   \vdots      \\
    \alpha-1      &   0          &    0          &    \cdots   &   \lambda-2\alpha    \\
\end{array}\right|\\
&=(\lambda-2\alpha)\phi(H_{n+1})-(\alpha-1)\left|
\begin{array}{cccccc}
       \alpha-1   &     &   \alpha-1    &    \cdots   &   0           \\
           0      &     &   \lambda-2\alpha   &    \cdots   &   0           \\
       \vdots     &     &   \vdots      &    \ddots   &   \vdots      \\
    \alpha-1      &     &    0          &    \cdots   &   \lambda-2\alpha    \\
\end{array}\right| -(\alpha-1)^2\phi(H_{n})+(-1)^{n+3}(\alpha-1)^{n+2}\\
&=(\lambda-2\alpha)\phi(H_{n+1})-(\alpha-1)^2\phi(H_{n})+(-1)^{n+3}(\alpha-1)^{n+2}-(\alpha-1)^2\phi(H_{n})+(-1)^{n+3}(\alpha-1)^{n+2}\\
&=(\lambda-2\alpha)\phi(H_{n+1})-2(\alpha-1)^2\phi(H_{n})+2(-1)^{n+3}(\alpha-1)^{n+2}.
\end{split}
\end{equation*}
}
This finishes the proof, since $n+1$ and $n+3$ have the same parity.
\end{proof}

Throughout the rest of the paper, we shall make use of the following notations:
\begin{equation}\label{kazz1} \Delta_{\lambda,\alpha} =\sqrt{(\lambda-4\alpha+2)(\lambda-2)} \qquad \text{and}  \qquad h(\lambda)_{\alpha}  = \frac{\lambda -\Delta_{\lambda,\alpha}}{2\alpha(\lambda-2)+2}. \end{equation}
\begin{prop}\label{gongshi} Let $n$ be any non-negative integer. After setting
\begin{equation*} s  =\frac{\lambda-2\alpha+\Delta_{\lambda,\alpha}}{2}, \qquad\text{and}  \qquad t= \frac{\lambda-2\alpha-\Delta_{\lambda,\alpha}}{2},
\end{equation*}
Equalities (i), (ii) and (iv) below hold for $\alpha \in [0,1)$. Equality (iii) holds for $\alpha \in (0,1)$.
\begin{itemize}
\item[$\mathrm{(i)}$]
$\phi(H_{n})=\Delta_{\lambda,\alpha}^{-1}(s^{n+1}-t^{n+1});$
\item[$\mathrm{(ii)}$]
$\phi(P_{n+1})=\Delta_{\lambda,\alpha}^{-1}((s+\alpha)^2s^{n}-(t+\alpha)^2t^{n});$
\item[$\mathrm{(iii)}$]
$\phi(B_{n+1})=$\small$\displaystyle \frac{1}{\Delta_{\lambda,\alpha}} \cdot \frac{\alpha}{\left( \alpha(\lambda-2)+1\right)} \left((s+\alpha)^2
\left(s+\frac{(1-\alpha)^2}{\alpha}\right)s^{n}-(t+\alpha)^2\left(t+\frac{(1-\alpha)^2}{\alpha}\right)t^{n}\right).$
\item[$\mathrm{(iv)}$] $\displaystyle \lim\limits_{n\rightarrow\infty} \frac{\phi(B_{n-1})}{\phi (P_n)} =
\lim\limits_{n\rightarrow\infty} \frac{\phi(H_{n-2})}{\phi (B_{n-1})}
=h(\lambda)_{\alpha} $.
\end{itemize}
\end{prop}

\begin{proof}
We start by noticing that $s$ and $t$ are the roots of the polynomial
$ x^2 -(\lambda-2\alpha)x +(1-\alpha)^2.$
Result (i) surely holds for $n \in\{0,1\}$. We now argue by induction on $n$. Suppose that $n \geqslant 1$. From \eqref{hn} and induction, it follows that
\begin{equation*}
\begin{split}
\phi(H_{n+1})&=(\lambda-2\alpha)\phi(H_{n})-(1-\alpha)^2\phi(H_{n-1})\\
           &=(\lambda-2\alpha)\Delta_{\lambda,\alpha}^{-1}(s^{n+1}-t^{n+1})-(1-\alpha)^2\Delta_{\lambda,\alpha}^{-1}(s^{n}-t^{n})\\
           &=\Delta_{\lambda,\alpha}^{-1}((\lambda-2\alpha)s-(1-\alpha)^2)s^{n}-\Delta_{\lambda,\alpha}^{-1}((\lambda-2\alpha)t-(1-\alpha)^2)t^{n}\\
         &=\Delta_{\lambda,\alpha}^{-1}(s^{n+2}-t^{n+2}).
\end{split}
\end{equation*}

Now that we know that (i) holds, Equalities (ii) and (iii)  come from Lemma \ref{e1234}(ii) and Lemma~\ref{e0} respectively, once we note  that, by definition, $s$ and $t$ both satisfy
\[ \lambda x + 2\alpha -1 = (x+\alpha)^2. \]
In order to prove (iv), the following identities turn out to be useful:
\begin{equation}\label{squash}
st=(1-\alpha)^2; \qquad \frac{t+\alpha}{a(\lambda-2)+1} = h(\lambda)_{\alpha},   \qquad \text{and} \qquad \frac{1}{(s+\alpha)^2}= \left(h(\lambda)_{\alpha} \right)^2.
\end{equation}
 It is also important to note that
 $ \lim_{n\rightarrow\infty} \left(t /s\right)^n =0,$ since $t<s$.
The cases $\alpha=0$ and $\alpha \in (0,1)$ will be dealt separately. If $\alpha=0$,
by definition we have $\phi(P_n)=\phi(B_n)=\phi(H_n)$. Therefore,
from (ii) we get
\[
\scalemath{1}{\lim\limits_{n\rightarrow\infty} \frac{\phi(B_{n-1})}{\phi (P_n)} =  \lim\limits_{n\rightarrow\infty} \frac{\phi(P_{n-1})}{\phi (P_n)} = \lim\limits_{n\rightarrow\infty} \frac{s^{n+1}-t^{n+1}}{s^{n+2}-t^{n+2}}
 =\frac{1}{s} \lim\limits_{n\rightarrow\infty} \frac{1-\left(\frac{t}{s}\right)^{n+1}}{1-\left(\frac{t}{s}\right)^{n+2}}=\frac{1}{s}=t=h(\lambda)_0},
\]
and, similarly,
   \[   \lim\limits_{n\rightarrow\infty} \frac{\phi(H_{n-2})}{\phi (B_{n-1})} =  \lim\limits_{n\rightarrow\infty} \frac{\phi(P_{n-2})}{\phi (P_{n-1})} =t=h(\lambda)_0 \]
as claimed. Let now $\alpha\in (0,1)$. Using (ii) and (iii), we obtain
 \[
\scalemath{.95}{ \begin{aligned}
\lim\limits_{n\rightarrow\infty} \frac{\phi(B_{n-1})}{\phi (P_n)} & = \lim\limits_{n\rightarrow\infty} \frac{\alpha}{ \alpha(\lambda-2)+1} \cdot \frac{(s+\alpha)^2 \left( s + \frac{(1-\alpha)^2}{\alpha}\right)s^{n-2} - (t+\alpha)^2 \left( t + \frac{(1-\alpha)^2}{\alpha}\right)^2t^{n-2}}{(s+\alpha)^2s^{n-1}-(t+\alpha)^2t^{n-1}} \\
&= \lim\limits_{n\rightarrow\infty} \frac{\alpha}{\alpha(\lambda-2)+1} \cdot \frac{1}{s} \cdot \frac{ \left( s + \frac{(1-\alpha)^2}{\alpha}\right) - \left(\frac{t+\alpha}{s+\alpha} \right)^2 \left( t + \frac{(1-\alpha)^2}{\alpha}\right)^2\left(\frac{t}{s} \right)^{n-2}}{1-\left(\frac{t+\alpha}{s+\alpha} \right)^2\left(\frac{t}{s} \right)^{n-1}}\\
& = \frac{\alpha}{ \alpha(\lambda-2)+1} \cdot \frac{1}{s} \cdot \left( s + \frac{(1-\alpha)^2}{\alpha}\right) \\
& = \frac{\alpha}{\alpha(\lambda-2)+1} \cdot \left(1+\frac{t}{\alpha}\right)\\
& =  \frac{t+\alpha}{\alpha(\lambda-2)+1}=h(\lambda)_{\alpha}.
\end{aligned} }
\]
For the last equality, we have used the second identity in \eqref{squash}. We now compute
\[ \scalemath{.95}{ \begin{aligned}
\lim\limits_{n\rightarrow\infty} \frac{\phi(H_{n-2})}{\phi (P_{n-1})} &= \lim\limits_{n\rightarrow\infty} \frac{(\alpha(\lambda-2)+1) (s^{n-1}-t^{n-1})}
{(s+\alpha)^2 (\alpha s +(1- \alpha)^2)s^{n-2} -  (t+\alpha)^2 (\alpha t +(1-\alpha)^2)t^{n-2}}\\[.8em]
          &= \lim\limits_{n\rightarrow\infty}  \frac{(\alpha(\lambda-2)+1) (s^{n-1}-t^{n-1})}
{(s+\alpha)^2 (\alpha s +st)s^{n-2} -  (t+\alpha)^2 (\alpha t +st)t^{n-2}}\\[.8em]
&= \lim\limits_{n\rightarrow\infty}  \frac{(\alpha(\lambda-2)+1) (1-\left(\frac{t}{s}\right)^{n-1})}
{(s+\alpha)^2 (\alpha  +t) -  (t+\alpha)^2 (\alpha  +s)\left(\frac{t}{s}\right)^{n-1}}\\
&= \frac{\alpha(\lambda-2)+1}{t+\alpha} \cdot \frac{1}{(s+\alpha)^2} = \frac{1}{h(\lambda)_{\alpha}}\cdot
\left(h(\lambda)_{\alpha} \right)^2 = h(\lambda)_{\alpha},
\end{aligned} } \]
where, for the last equality, we have used the second and the third identity in \eqref{squash}.
\end{proof}

According to  \cite{hof-smi}, an {\it internal path} of a graph $G$ is a walk $v_0 v_1 \dots v_k$ (here $k \geqslant 1$), where the vertices $v_1, \dots, v_k$ are pairwise distinct, $d(v_0) > 2$, $d(v_k) > 2$ and $d(v_i) = 2$ whenever $0 < i < k$. We say that an internal path is of {\em type I} (resp., {\em type II}) if $v_0=v_k$ (resp., $v_0\not=v_k$) (see Fig.~\ref{fig100}).

Subdividing an edge belonging to an internal paths has well-known $A$-spectral consequences (see \cite{hof-smi} or [1, p. 79]).  The impact on the $A_{\alpha}$-spectral radius when internal paths are involved is stated in \cite{li-chen-meng2}. Proposition~\ref{alpha-internal} deals with all cases. We recall that the {\em double snake of order $n \geqslant 6$} is the graph $W_n$ depicted in Fig.~\ref{fig1} and~\ref{fig100} containing  an internal path of type II $v_0 \dots  v_{n-5}$ such that $d(v_0)=d(v_{n-5})=3$.
\medskip

  \begin{figure}[h]
\begin{center}
\includegraphics[width=0.7\textwidth]{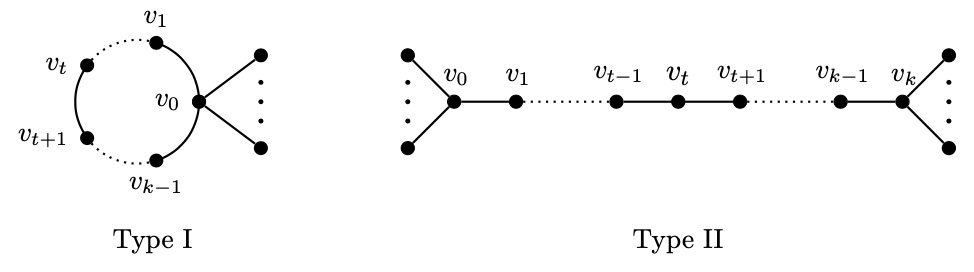}
  \caption{ \label{fig100}  \small The two types of internal path}
\end{center}
\end{figure}

\begin{prop}\label{alpha-internal}
Let $uv$ be an edge of the connected graph $G$, and let $G_{uv}$ be the graph obtained from $G$ by subdividing the edge $uv$ of $G$. Set $\alpha \in [0,1)$.
\begin{itemize}
\item[$\mathrm{(i)}$]
$\rho\!_{_{A_\alpha}}\!(C_n)=2$ and $\rho\!_{_{A_0}}\!(W_n)=2$;
\item[$\mathrm{(ii)}$]
If $(G,\alpha) \neq (C_n,\alpha)$ and $uv$ is not in an internal path of $G$, then $\rho\!_{_{A_{\alpha}}}\!(G_{uv})>\rho\!_{_{A_{\alpha}}}\!(G)$;
\item[$\mathrm{(iii)}$]
If $(G,\alpha) \neq (W_n,0)$ and $uv$ belongs to an internal path of $G$, then $\rho\!_{_{A_{\alpha}}}\!(G_{uv})<\rho\!_{_{A_{\alpha}}}\!(G)$.
\end{itemize}
\end{prop}

\begin{proof}
It is straightforward to show that the $A_\alpha$-spectral radius of a regular graph $G$ is the degree of $G$ \cite[Section~3.2]{niki}; hence,  $\rho_{A_\alpha}(C_n)=2$. The latter of (i) follows from \cite{hof-smi}.

In the hypotheses of (ii), it is not hard to show that $G_{uv}$ contains a proper subgraph isomorphic to $G$. Therefore, by Lemma~\ref{alpha-delta}(ii) we get $\rho\!_{_{A_{\alpha}}}\!(G_{uv})>\rho\!_{_{A_{\alpha}}}\!(G)$.

For (iii), we refer the reader to the proof of Lemma 1.1 in \cite{li-chen-meng2}, with the warning that in such proof
the authors assume that $G$ properly contains a double snake. What is really crucial is their argument is that $\rho\!_{_{A_{\alpha}}}\!(G)>2$, and this is true whenever $G$ contains an internal path of {\em any} type
and $(G,\alpha) \neq (W_n,0)$. In fact, from Lemmas 3.1(ii) and 3.2(ii) in \cite{WWXB},  it follows that  $\rho\!_{_{A_{\alpha}}}\!(W_n)>2$ if and only if $\alpha \neq 0$; moreover, if $G$ contains an internal path of type I,  {\em a fortiori} it contains a cycle $C_r$ as a proper subgraph; therefore, using (i) together with Lemma~\ref{alpha-delta}(ii), we immediately get $\rho\!_{_{A_{\alpha}}}\!(G)> \rho\!_{_{A_{\alpha}}}\!(C_r) =2$.
\end{proof}

\subsection{Limit points for the $A_\alpha$-spectral radius of compound graphs}\label{sub2}

 Let $v$ be an end vertex of the path $P_n$, and let $u$ be a vertex of a graph $G$ vertex-disjoint with respect to $P_n$. We denote by $G_u(P_n)$ the graph $Gu \!:\!vP_n$. Recall that $V(G_u(P_n))= V(G) \cup V(P_n)$ and $E(G_u(P_n))=E(G) \cup E(P_n) \cup \{uv\}$.

\begin{prop}\label{alpha-gupn}
\vspace{1em} The $A_{\alpha}$-spectral radius of the graph sequence $\{G_u(P_n)\}_{n \in \N}$ has a limit point $\chi_u(G) \geqslant 2$. If $\chi_u(G)>2$, then $\chi_u(G)$ is the largest root of the equation

 $$\left(1-\alpha\cdot h(\lambda)_{\alpha} \right)\phi(G)-
\left(\alpha-(2\alpha-1)\cdot h(\lambda)_{\alpha}\right)\phi(G)_u=0,
$$
where $h(\lambda)_{\alpha}$ is defined in \eqref{kazz1}.
\end{prop}
\begin{proof}
The inequalities
$$\rho\!_{_{A_{\alpha}}}\!(G_u(P_n))<\rho\!_{_{A_{\alpha}}}\!(G_u(P_{n+1})) \qquad \text{and} \qquad \rho\!_{_{A_{\alpha}}}\!(G_u(P_n))\leqslant\Delta(G_u(P_n))$$
 both come from Lemma \ref{alpha-delta}. Altogether, they imply that $\lim_{n\rightarrow\infty}\rho\!_{_{A_{\alpha}}}\!(G_u(P_n))=\chi_u(G)$ exists. If $G_u(P_n)$ is a path, then by Lemma \ref{alpha-pn}, we get $\chi_u(G)=2$. Since $G_u(P_n)$ properly contains $P_n$,
then $\rho\!_{_{A_{\alpha}}}\!(G_u(P_n)) > \rho\!_{_{A_{\alpha}}}\!(P_n)$; hence,  $\chi_u(G) \geqslant 2$.

Suppose now $\chi_u(G)>2$. From Lemma \ref{alpha-g1g2} we get
\begin{equation}\label{vaffax}
\begin{split}
\phi(G_u(P_n))&=\phi(G)\phi(P_n)-\alpha\phi(G)_u\phi(P_n)-\alpha\phi(G)\phi(B_{n-1})+(2\alpha-1)\phi(G)_u\phi(B_{n-1})\\[.2cm]
              &=\phi(P_n)\left( \left(1-\alpha \frac{\phi(B_{n-1})}{\phi(P_n)}     \right)\phi(G) - \left(\alpha -(2\alpha-1) \frac{\phi(B_{n-1})}{\phi(P_n)}   \right) \phi(G)_u \right).
\end{split}
\end{equation}
Since $\chi_u(G)>2$ and $\rho\!_{_{A_{\alpha}}}\!(P_n) \leqslant 2 $, then $\chi_u(G)$ is the largest positive root of the following equation:
\[ \lim\limits_{n\rightarrow\infty}\left( \left(1-\alpha \frac{\phi(B_{n-1})}{\phi(P_n)}     \right)\phi(G) - \left(\alpha -(2\alpha-1) \frac{\phi(B_{n-1})}{\phi(P_n)}   \right) \phi(G)_u \right)=0,\]
and Proposition~\ref{gongshi}(iv) ensures that
$\displaystyle \lim\limits_{n\rightarrow\infty} \phi(B_{n-1})/\phi(P_n)=  h(\lambda)_{\alpha}$.
\end{proof}

\begin{cor}\label{K13P5Pn_1} Let $u$ be the vertex in of degree 3 in $K_{1,3}$. Then,
$$\lim\limits_{n\rightarrow\infty}\rho\!_{_{A_{\alpha}}}\!((K_{1,3})_u(P_n))=\frac{1}{2}\left(5\alpha+3\sqrt{2-4\alpha+3\alpha^2}\right).$$
\end{cor}
\begin{proof}
Lemma~\ref{alpha-delta} and a direct calculation guarantee that, for all $n\geqslant 2$,
$$\rho\!_{_{A_{\alpha}}}\! (K_{1,3}(P_n)) \geqslant \rho\!_{_{A_0}}\! (K_{1,3}(P_n)) \geqslant \rho\!_{_{A_0}}\! (K_{1,3}(P_2)) >2.$$
Therefore, from Proposition~\ref{alpha-gupn}, it follows that $\lim\limits_{n\rightarrow\infty}\rho\!_{_{A_{\alpha}}}\!(K_{1,3}(P_n))$ is the largest root of the following equation:

$$\left(1-\alpha h(\lambda)_{\alpha} \right)\phi(K_{1,3})-
\left(\alpha-(2\alpha-1)\cdot h(\lambda)_{\alpha} \right)\phi(K_{1,3})_u =0.$$
 Plugging into the equation above
$$  \phi(K_{1,3})=(\lambda-\alpha)^2(\lambda^2-4\alpha \lambda+6\alpha-3) \qquad \text{and}  \qquad \phi(K_{1,3})_u=(\lambda-\alpha)^3,$$
we find that the largest root of the above equation is $\displaystyle \frac{1}{2}\left(5\alpha+3\sqrt{2-4\alpha+3\alpha^2}\right)$.
\end{proof}

Let $2P_n$ be the disjoint union of two copies of $P_n$ and let $u_1,u_2 \in V(2P_n)$ be end-vertices belonging to different components. For every non-trivial connected graph $G$ and every $u \in V(G)$, we consider the graph $G_u(P_n,P_n)$ obtained by adding to $G \cup 2P_n$ the edges $uu_1$ and $uu_2$.

\begin{prop}\label{alpha-gupnpn}
The $A_{\alpha}$-spectral radius of the graph sequence $\{G_u(P_n,P_n)\}_{n \in \N}$ has a limit point $\chi'_u(G) \geqslant 2$. If $\chi'_u(G)>2$, then $\chi'_u(G)$ is the largest root of the equation $\Theta(\lambda)_{G,u,\alpha, \infty}=0$, where
\begin{equation*}
\Theta(\lambda)_{G,u,\alpha, \infty}= \left(1-\alpha h(\lambda)_{\alpha}\right)\left(\phi(G)(1-\alpha h(\lambda)_{\alpha})-2\alpha\phi(G)_u+2(2\alpha-1)\phi(G)_uh(\lambda)_{\alpha}\right),\end{equation*}
and $h(\lambda)_{\alpha}$ is defined in \eqref{kazz1}.
\end{prop}

\begin{proof}
By Lemma~\ref{alpha-delta}(i), $\{G_u(P_n,P_n)\}_{n \in \N}$ is a nested sequence of graphs whose $A_{\alpha}$-spectral radius is limited above by $\Delta(G)+2$. Therefore,
$\lim_{n\rightarrow\infty}\rho\!_{_{A_{\alpha}}}\!(G_u(P_n,P_n))=\chi'_u(G)$ exists, and $\chi'_u(G) \geqslant \chi_u(G) \geqslant 2$, since $G_u(P_n,P_n)$ properly contains $G_u(P_n)$, and Proposition~\ref{alpha-gupnpn} holds.

Now, by applying twice Lemma \ref{alpha-g1g2} and using \eqref{e2} and \eqref{vaffax},
we get
\begin{equation*}
\begin{split}
\phi(G_u(P_n,P_n)) &= \phi(G_u(P_n))\phi(P_n)-\alpha\phi(G_u)\phi(B_n)\phi(P_n) \\   & \phantom{=}
-\alpha\phi(G_u(P_n))\phi(B_{n-1})+(2\alpha-1)\phi(G)_u\phi(B_n)\phi(B_{n-1}).
\\[.8em]
&= \phi(G)(\phi(P_n)-\alpha\phi(B_{n-1}))^2-2\alpha\phi(G)_u\phi(P_n)(\phi(P_n)-\alpha\phi(B_{n-1}))\\
                  &\phantom{=}
                  +2(2\alpha-1)\phi(G)_u\phi(B_{n-1})(\phi(P_n)-\alpha\phi(B_{n-1}))
                  \end{split}
\end{equation*}
which is also equal to $\phi(P_n)^2 \Theta(\lambda)_{G,u,\alpha,n}$, where
{ \small$$ \Theta(\lambda)_{G,u,\alpha,n}= \left(1-\alpha\frac{\phi(B_{n-1})}{\phi(P_n)}\right)\left(\phi(G)\left(1-\alpha\frac{\phi(B_{n-1})}{\phi(P_n)}\right)
                     -2\alpha\phi(G)_u +2(2\alpha-1)\phi(G)_u\frac{\phi(B_{n-1})}{\phi(P_n)}\right).$$}
By Lemma~\ref{alpha-delta}, no roots of $(\phi(P_n))^2$ are larger than $2$; therefore,
if $\chi'(G)>2$, such number is the largest positive root of
$$ \lim\limits_{n\rightarrow\infty} \Theta(\lambda)_{G,u,\alpha, n}= \Theta(\lambda)_{G,u,\alpha,\infty},$$
where the equality comes from Proposition~\ref{gongshi}(iv).
\end{proof}

\medskip
%
%
%

  \begin{figure}[h]
\begin{center}
\includegraphics[width=0.6\textwidth]{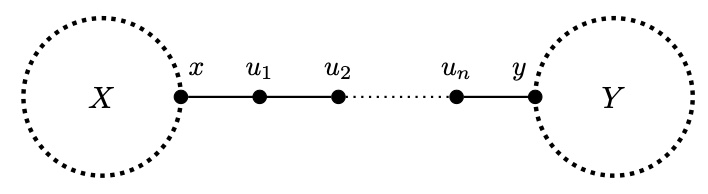}
  \caption{ \label{fig3}  \small The graph $XY(x,y;n)$}
\end{center}
\end{figure}

\begin{prop}\label{alpha-gxy} Let $X$ and $Y$
be two vertex-disjoint connected graphs, and let $G_n=XY(x,y;n)$ be the graph  obtained by joining $x \in V(X)$ and $y \in V(Y)$  by a path of length $n+1$ (see Fig.~\ref{fig3}). Then,
\begin{equation}\label{mink}
\lim\limits_{n\rightarrow\infty}\rho\!_{_{A_{\alpha}}}\!(G_n)=\max \left\{ \lim\limits_{n\rightarrow\infty}\rho\!_{_{A_{\alpha}}}\!(X_x(P_n)),\lim\limits_{n\rightarrow\infty}\rho\!_{_{A_{\alpha}}}\!(Y_y(P_n)) \right\}.
\end{equation}
\end{prop}

\begin{proof}
The existence of $\lim_{n\rightarrow\infty}\rho\!_{_{A_{\alpha}}}\!(X_x(P_n))$ and $\lim_{n\rightarrow\infty}\rho\!_{_{A_{\alpha}}}\!(Y_y(P_n))$ follows from Proposition \ref{alpha-gupn}(i). We refer the reader to Fig.~\ref{fig3} for notation. If the path $u_1u_2\dots u_n$ is contained in an internal path of $G_n$, then from Lemma \ref{alpha-delta}(i) and Proposition \ref{alpha-internal}(iii) we know that $\lim_{n\rightarrow\infty}\rho\!_{_{A_{\alpha}}}\!(G_n)$ exists. Otherwise, $G_n$ is either of type $X_x(P_{n+\lvert V(Y)\rvert})$ or $Y_y(P_{n+\lvert V(X) \rvert})$. Thus, by Proposition \ref{alpha-gupn},
$\lim_{n\rightarrow\infty} \rho\!_{_{A_{\alpha}}}\!(G_n)$  exists as well. Since the smallest limit point for the $A_{\alpha}$-spectral radius is $2$ (see \cite[Theorem~1.1]{WWXB}), we get
\begin{multline}\label{kxk}  2 \leqslant \min \left\{ \lim\limits_{n\rightarrow\infty}\rho\!_{_{A_{\alpha}}}\!(X_x(P_n)),\lim\limits_{n\rightarrow\infty}\rho\!_{_{A_{\alpha}}}\!(Y_y(P_n)) \right\} \\
\leqslant \max \left\{ \lim\limits_{n\rightarrow\infty}\rho\!_{_{A_{\alpha}}}\!(X_x(P_n)),\lim\limits_{n\rightarrow\infty}\rho\!_{_{A_{\alpha}}}\!(Y_y(P_n)) \right\} \leqslant \lim\limits_{n\rightarrow\infty}\rho\!_{_{A_{\alpha}}}\!(G_n), \end{multline}
where the third inequality comes from Lemma~\ref{alpha-delta}(ii). If $\lim_{n\rightarrow\infty}\rho\!_{_{A_{\alpha}}}\!(G_n)=2$,  Inequalities \eqref{kxk} imply \eqref{mink}.

\smallskip
If instead $\lim_{n\rightarrow\infty}\rho\!_{_{A_{\alpha}}}\!(G)>2$, we use the fact that $G_n= (X_x(P_n))u_n\!:\! yY$. Applying
Lemma \ref{alpha-g1g2}, we obtain:
\begin{equation*}
\scalemath{.95}{\phi(G_n)=\phi(X_x(P_n))\phi(Y)-\alpha\phi(X_x(P_n))\phi (Y)_y-\alpha\phi(X_x(P_n))_{u_n}\phi(Y)+(2\alpha-1) \phi(X_x(P_n))_{u_n}   \phi(Y)_y,     } \end{equation*}
where $$ \phi(X_x(P_n)) = \phi(X)\phi(P_n)-\alpha \phi(X)_x \phi(P_n) - \alpha\phi(X)\phi (B_{n-1})+(2\alpha-1)\phi(X)_x\phi(B_{n-1}) $$
and \[ \scalemath{.95}{ \phi(X_x(P_n))_{u_n} =\phi(X)\phi(B_{n-1})-\alpha \phi(X)_x \phi(B_{n-1})-\alpha\phi(X)\phi(H_{n-2})
       +(2\alpha-1)\phi(X)_x\phi(H_{n-2}).} \]

A straightforward algebraic manipulation shows that $\phi(G_n) = \phi (P_n) \cdot \Gamma_n$, where

\[  \scalemath{.95}{  \begin{aligned}
\Gamma_n =  \left(  \left( 1-\alpha \frac{\phi(B_{n-1})}{\phi(P_{n})} \right)  \phi(X) -\left(
    \alpha -(2\alpha-1)  \frac{\phi(B_{n-1})}{\phi(P_{n})} \right) \phi(X)_x \right) \left( \phi(Y)-\alpha\phi(Y)_y \right) \phantom{+++++}\\[0.6em]
       +\frac{\phi(B_{n-1})}{\phi(P_{n})}  \left(  \left( 1-\alpha \frac{\phi(H_{n-2})}{\phi(B_{n-1})} \right)  \phi(X) -\left(
    \alpha -(2\alpha-1)  \frac{\phi(H_{n-2})}{\phi(B_{n-1})} \right) \phi(X)_x \right) \left( (2\alpha-1)\phi(Y)_y-\alpha\phi(Y) \right)
      \end{aligned} }\]

By Proposition \ref{gongshi}(iv), we know that
$$\lim_{n\rightarrow\infty}\phi(B_{n-1})/\phi(P_n)=\lim_{n\rightarrow\infty}\phi(H_{n-2})/\phi(B_{n-1})
=h(\lambda)_{\alpha}.$$
It is now elementary to check that $\lim_{n\rightarrow\infty}\Gamma_{n}$ is the product of
$$\left(1-\alpha\cdot h(\lambda)_{\alpha} \right)\phi(X)-
\left(\alpha-(2\alpha-1)\cdot h(\lambda)_{\alpha}\right)\phi(X)_x$$
and
$$\left(1-\alpha\cdot h(\lambda)_{\alpha} \right)\phi(Y)-
\left(\alpha-(2\alpha-1)\cdot h(\lambda)_{\alpha}\right)\phi(Y)_y,$$
whose maximum roots, by
Proposition~\ref{alpha-gupn}, are $\lim\limits_{n\rightarrow\infty}\rho\!_{_{A_{\alpha}}}\!(X_x(P_n))$ and $\lim\limits_{n\rightarrow\infty}\rho\!_{_{A_{\alpha}}}\!(Y_y(P_n))$ respectively. This ends the proof.
\end{proof}

\section{Remarks}

In this paper, we summarize the results on Hoffman program of graphs with respect to the adjacency, the Laplacian, the signless Laplacian, the Hermitian adjacency and skew-adjacency matrix of  graphs. As well, the tensors of hypergraphs are also involved.  Moreover, we put forward to some related problems for further study. Particularly, we obtain new results about the Hoffman program with relation to the $A_\alpha$-matrix.
 
As already observed in Section 9, the $A_{\alpha}$-matrix of a graph $G$ encodes the properties of $A(G)$, $Q(G)$, and $L(G)$. 
Therefore, it is reasonable to focus efforts to obtain a formula for the $A_{\alpha}$-counterpart of   
the $A$-limit point $\sqrt{2+\sqrt{5}}$, and of the number $\omega$ which is both an $L$- and a $Q$-limit point. We have made some progress in this direction, but the presence of the variable $\alpha$ makes the $A_\alpha$-polynomials of graphs quite hard to manipulate. We will discuss this matter in a forthcoming paper, containing further advances on the Hoffman program for the $A_\alpha$-matrix.

\section*{Acknowledgments}
The first  author is supported for this research by the National Natural Science Foundation of China (No. 11971274).

\small{

}

\end{document}